\documentclass[english,12pt]{amsart}

\usepackage{amssymb}

\usepackage[utf8]{inputenc}
\usepackage[T1]{fontenc}
\usepackage[english]{babel}

\usepackage[inline]{enumitem}
\setenumerate[1]{font=\color{black}, label=\textbf{\arabic*.}}

\usepackage{yhmath}
\usepackage{enumitem}
\setlist{topsep=0pt} % to remove the linespread before and after a list.

\makeatletter
\def\namedlabel#1#2{\begingroup
	#2%
	\def\@currentlabel{#2}%
	\phantomsection\label{#1}\endgroup
}
\makeatother
\usepackage{euscript}

\usepackage{fouriernc}%%%%%%%%%%%%%% mettre les 3 lignes
\usepackage[scaled=0.875]{helvet}
\usepackage{courier}

\usepackage{ marvosym }

\usepackage[babel=true]{csquotes}

\usepackage[marginparwidth=2cm]{geometry}
\usepackage[usenames,dvipsnames,svgnames,table]{xcolor}
\usepackage[colorlinks=true,urlcolor=DarkBlue,linkcolor=DarkRed,citecolor=DarkGreen]{hyperref}

\usepackage[hyperref=true,style=alphabetic,citestyle=alphabetic,sorting=nyt,backend=bibtex,maxnames=100,doi=false,isbn=false,url=false,backref=true,backrefstyle=all+,backrefsetstyle=setonly]{biblatex}
\bibliography{zariskibib}

\usepackage[normalem]{ulem}
\usepackage[marginpar]{todo}

\newcommand{\Rd}{\color{RedOrange}}

\newcommand{\Bk}{\color{black}}

\usepackage{nicefrac}
\usepackage{tikz}
\usepackage{tikz-cd}
\usepackage{fp}
\usetikzlibrary{calc}
\usetikzlibrary{shapes,arrows}
\usetikzlibrary{fit,positioning}
\usetikzlibrary{patterns,decorations.pathreplacing}
\makeatletter
\def\calcLength(#1,#2)#3{%
	\pgfpointdiff{\pgfpointanchor{#1}{center}}%
	{\pgfpointanchor{#2}{center}}%
	\pgf@xa=\pgf@x%
	\pgf@ya=\pgf@y%
	\FPeval\@temp@a{\pgfmath@tonumber{\pgf@xa}}%
	\FPeval\@temp@b{\pgfmath@tonumber{\pgf@ya}}%
	\FPeval\@temp@sum{(\@temp@a*\@temp@a+\@temp@b*\@temp@b)}%
	\FProot{\FPMathLen}{\@temp@sum}{2}%
	\FPround\FPMathLen\FPMathLen5\relax
	\global\expandafter\edef\csname #3\endcsname{\FPMathLen}
}
\makeatother

\usepackage{xkeyval}
\usepackage{moreverb}
\usepackage{epic}
\usepgfmodule{shapes,plot,decorations}
\usepackage{accents}
\usepackage{epsfig}
\usepackage{pinlabel}
\usepackage{subfig}

\newcommand{\N}{\mathbb{N}}
\newcommand{\Z}{\mathbb{Z}}

\newcommand{\R}{\mathbb{R}}

\newcommand{\Cb}{\mathbb{C}} 
\newcommand{\Hb}{\mathbb{H}}

\newcommand{\Pb}{\mathbb{P}}

\newcommand{\GL}{\mathrm{GL}}
\newcommand{\SL}{\mathrm{SL}}

\newcommand{\SO}{\mathrm{SO}}

\newcommand{\G}{\Gamma}

\newcommand{\Cc}{\mathcal{C}}

\newcommand{\Ec}{\mathcal{E}}

\renewcommand{\O}{\Omega}

\usepackage{lipsum}

\newcommand{\Sp}{\textrm{Sp}}

\newcommand{\Id}{\mathrm{Id}}
\DeclareMathOperator{\id}{id}

\renewcommand{\leq}{\leqslant}
\renewcommand{\geq}{\geqslant}
\renewcommand{\varnothing}{\emptyset}

\usepackage{mathrsfs}  
\newcommand{\Cart}{\mathscr{A}}

\usepackage{aurical}

\theoremstyle{plain}
\newtheorem{theorem}{Theorem}[section]

\newtheorem{proposition}[theorem]{Proposition}
\newtheorem{corollary}[theorem]{Corollary}
\newtheorem{lemma}[theorem]{Lemma}

\theoremstyle{definition}
\newtheorem{example}[theorem]{Example}
\newtheorem{definition}[theorem]{Definition}

\theoremstyle{remark}
\newtheorem{remark}[theorem]{Remark}

\title{Zariski-Closures of Linear Reflection Groups}

\author[J. Audibert]{Jacques Audibert}
\address{Max Planck Institute for Mathematics in the Sciences, 04299 Leipzig, Germany}
\email{audibert.j@outlook.fr}

\author[S. Douba]{Sami Douba}
\address{Institut des Hautes \'Etudes Scientifiques, 
Universit\'e Paris-Saclay, 35 route de Chartres, 91440 Bures-sur-Yvette, France}
\email{douba@ihes.fr}

\author[G.-S. Lee]{Gye-Seon Lee}
\address{Department of Mathematical Sciences and Research institute of Mathematics, Seoul National University, Seoul 08826, South Korea}
\email{gyeseonlee@snu.ac.kr}

\author[L. Marquis]{Ludovic Marquis}
\address{Universit\'e de Rennes, CNRS, IRMAR - UMR 6625, 35000 Rennes, France}
\email{ludovic.marquis@univ-rennes.fr}

\subjclass[2020]{22E40, 20F55}
\usepackage{marginnote}
\newcounter{notes}%
\newcommand{\note}[1]{
	\refstepcounter{notes}  % \refstepcounter{notes} fait comme \addtocounter{notes}{1}, mais permet en principe d'utiliser \label{...} dans la note.
	\nolinebreak
	$\hspace{-5pt}{}^{\text{\tiny \rm \arabic{notes}}}$
	\marginpar{\tiny \arabic{notes}) #1}
	}

\begin{document}

\begin{abstract}
%Vinberg's representations of coxeter groups give deformation spaces of reflection polytope groups. It was demonstrated by Benoist and de la Harpe that if the corresponding Cartan matrix is symmetrical with signature $(p,q,r)$, then the reflection group is Zariski dense in $\SO(p,q)$. We demonstrate that if the Cartan matrix is not symmetrical and of rank $n$, then the reflection group is Zariski dense in $\SL(n,\R)$.\cite{Benoist-delaHarpe}
We give necessary and sufficient conditions for a linear reflection group in the sense of Vinberg to be Zariski-dense in the ambient projective general linear group. As an application, we show that every irreducible right-angled Coxeter group of  rank $N \geq 3$ virtually embeds Zariski-densely in $\mathrm{SL}_n(\mathbb{Z})$ for all $n \geq N$. This allows us to settle the existence of Zariski-dense surface subgroups of $\mathrm{SL}_n(\mathbb{Z})$ for all $n \geq 3$. Among the other applications are examples of Zariski-dense one-ended finitely generated subgroups of $\mathrm{SL}_n(\mathbb{Z})$ that are not finitely presented for all $n \geq 6$.
\end{abstract}

\keywords{Discrete subgroups of Lie groups, Coxeter groups, reflection groups, thin subgroups}

\maketitle

\setlength{\parskip}{0em}
\tableofcontents
\setlength{\parskip}{1em}

\section{Introduction}

%Writing the reflections as $\rho(s)=\mathrm{Id}-\alpha_s\otimes v_s$ with $\alpha_s\in V^*$ and $v_s\in V$,  the Cartan matrix of $\rho$ is defined by $\Cart=(\alpha_s(v_t))_{st}$. The Cartan matrix is only well-defined up to conjugation by a diagonal matrix with positive coefficients. Hence, we say that $\Cart$ is symmetrizable if it is conjugate to a symmetric matrix by a diagonal matrix with positive coefficients. Generalizing Vinberg's idea, one would like to read off the properties of $\rho$ from its Cartan matrix. This is the content of our main theorem.

Due to the profound interplay between their geometric, combinatorial, and algebraic properties, reflection groups in real hyperbolic spaces $\mathbb{H}^d$ have long constituted a rich source of examples in hyperbolic geometry and the theory of Fuchsian and Kleinian groups. For $d=2$, their study dates back to 19th-century work of von Dyck, Klein, and Poincaré. Later, a right-angled reflection group in $\mathbb{H}^3$ gave rise to the first example of a closed hyperbolic $3$-manifold~\cite{lobell1931beispiele}. Andreev \cite{Andreev, Andreev2, roeder2007andreev} would go on to demonstrate the vast abundance of finite-covolume reflection groups in $\mathbb{H}^3$.

By regarding real hyperbolic spaces $\mathbb{H}^d$ in the hyperboloid model, reflection groups in~$\mathbb{H}^d$ can be viewed as instances of linear reflection groups in finite-dimensional real vector spaces. Here, a {\em reflection} of such a vector space $V$ is an order-2 endomorphism of $V$ fixing a linear hyperplane pointwise. In \cite{the_bible}, Vinberg gave necessary and sufficient conditions for the translates of a convex polyhedral cone $\widetilde{\Delta}\subset V$ under a group $\Gamma < \mathrm{GL}(V)$ generated by reflections in the walls of $\widetilde{\Delta}$ to ``tile.'' As demonstrated by Vinberg, the (necessarily discrete) subgroup $\Gamma$ is then naturally isomorphic to the Coxeter group determined by the combinatorics of this tiling; see \S \ref{sec:Vinberg-theory} for precise definitions and statements. 

From this perspective, reflection groups in $\mathbb{H}^d$ are distinguished within the overall family of linear reflection groups in that that they preserve a nondegenerate quadratic form of signature $(d,1)$ on $V=\mathbb{R}^{d+1}$. However, while there are abstract Coxeter groups that cannot be realized as reflection groups in $\mathbb{H}^d$ for any~$d$ (including some Gromov-hyperbolic examples \cite{felikson2005series, lee2019ads}), every finitely generated Coxeter group can be realized as a linear reflection group in an appropriate $V$ in the above sense \cite{Bourbaki_LiegroupsandLiealgebras46, the_bible}. Furthermore, as will be exploited in the sequel, a single Coxeter group (indeed, even a reflection group in $\mathbb{H}^d$) often admits many realizations as a linear reflection group (which may preserve a form that is not Lorentzian, or fail to preserve any form whatsoever). For these reasons and others, Vinberg's theory has emerged as an indispensable source of examples of infinite-covolume discrete subgroups of higher-rank semisimple Lie groups \cite{GyeSeonLudovic_DiscreteCoxeterGroups}.

Note that any subgroup of $\mathrm{GL}(V)$ generated by reflections is in fact contained in the subgroup $\mathrm{SL}^\pm(V)$ of endomorphisms of determinant $\pm 1$. Our main result provides sufficient (and evidently necessary) conditions for the Zariski-closure of a linear reflection group in $V$ in the sense of Vinberg to be the entire group $\mathrm{SL}^\pm(V)$.

\begin{theorem}\label{thm:coxeter}
Let $W$ be a finitely generated Coxeter group that is not virtually abelian and $\rho: W \rightarrow \mathrm{GL}(V)$ a representation of $W$ as a reflection group (see Def. \ref{def:refl_gp}). Suppose that $\rho$ is irreducible. 
\begin{enumerate}
\item\label{BdlH} If $\rho$ preserves a nonzero symmetric bilinear form $f$ on $V$, then the Zariski-closure of $\rho(W)$ is   the orthogonal group $\mathrm{O}_f(V)$ of $f$.

\item\label{ADLM} Otherwise, the Zariski-closure of $\rho(W)$ is $\mathrm{SL}^\pm(V)$.
\end{enumerate}
\end{theorem}

Item (\ref{BdlH}) in Theorem \ref{thm:coxeter} was proved by Benoist and de la Harpe \cite{BenoistHarpe_Zariski} assuming that~$\rho$ is the so-called {\em geometric representation} of $W$, though Item (\ref{BdlH}) in full generality follows from their proof. Our contribution to Theorem \ref{thm:coxeter} is Item (\ref{ADLM}), but our argument will treat simultaneously Items~(\ref{BdlH}) and (\ref{ADLM}), hence giving an alternative proof of Item (\ref{BdlH}). 

Theorem \ref{thm:coxeter} was proved by the fourth author \cite[Thm.~B]{cox_in_hil} under a certain {\em $2$-perfectness} condition. A particular instance where the latter condition is satisfied is when $\rho(W)$ preserves and acts cocompactly on a properly convex domain in $\mathbb{P}(V)$, in which case Theorem \ref{thm:coxeter} in fact already follows from previous work of Benoist \cite{benoist2003convexes}. However, this $2$-perfectness assumption can only be satisfied if the virtual cohomological dimension of~$W$ is equal to $\dim(\mathbb{P}(V))$ or to $\dim(\mathbb{P}(V))-1$.

Theorem \ref{thm:coxeter} is useful because it is in practice easy to verify if a representation of a Coxeter group as a reflection group is irreducible or preserves a form (see Theorem \ref{thm:tits-vinberg2}).

\subsection{Applications to thin subgroups of $\mathrm{SL}_n(\mathbb{Z})$}
\noindent One of our motivations in proving Theorem \ref{thm:coxeter} was to produce new examples of finitely generated Zariski-dense infinite-index subgroups of $\mathrm{SL}_n(\mathbb{Z})$, or {\em thin subgroups} in the language of Sarnak \cite{sarnak2014notes}. There has recently been an increased interest in such subgroups, owing in part to the fact that their congruence Schreier graphs mimic those of $\mathrm{SL}_n(\mathbb{Z})$ itself~\cite{golsefidy2012expansion, breuillard2015approximate}.

By leveraging a straightforward criterion due to Vinberg for determining whether an irreducible representation of a Coxeter group as a reflection group is integral (see Lemma \ref{lem:cyclic-products-integers}), we show that any irreducible right-angled Coxeter group of rank $N \geq 3$ virtually embeds as a thin\footnote{Note that $\mathrm{SL}_n(\mathbb{Z})$ is not abstractly commensurable to a Coxeter group for $n \geq 3$ since, for instance, infinite Coxeter groups lack Kazhdan's property (T) \cite{cox_gp_dont_have_T}.} subgroup of $\mathrm{SL}_n(\mathbb{Z})$ for all $n \geq N$.

%As applications of our main result, we construct new examples of thin subgroups of $\mathrm{SL}_n(\mathbb{Z})$. Thin subgroups are Zariski-dense yet infinite index subgroups of $\mathrm{SL}_n(\mathbb{Z})$.
\begin{theorem}\label{thm:virtually-Zariski-dense}
Let $W$ be an irreducible right-angled Coxeter group of rank $N$ with $3 \leq N < \infty$.  For each $n \geq N$, there is a finite-index reflection subgroup~$\Gamma_n$ of~$W$ and a  representation $\rho_n : \Gamma_n \rightarrow \mathrm{GL}_n(\mathbb{R})$ as a reflection group that embeds $\Gamma_n$ as a Zariski-dense subgroup of $\mathrm{SL}_n^\pm (\mathbb{Z})$.
Moreover, if the Coxeter diagram of $W$ is not a tree,\footnote{If the Coxeter diagram of~$W$ is a tree, then every representation $\rho: W \rightarrow \mathrm{GL}(V)$ of $W$ as a reflection group preserves a nonzero symmetric bilinear form on $V$.} then we may take $\Gamma_{N} = W$.
\end{theorem}

That one passes to finite-index reflection subgroups in the statement of Theorem \ref{thm:virtually-Zariski-dense} is not merely a matter of convenience, since a subgroup of $\mathrm{SL}^\pm(V)$ generated by $N$ reflections will never act irreducibly on $V$, let alone be Zariski-dense in $\mathrm{SL}^\pm(V)$, if $\dim(V) > N$.

In the case that the Coxeter group $W$ in the statement of Theorem \ref{thm:virtually-Zariski-dense} is moreover Gromov-hyperbolic, then it follows from our proof together with work of Danciger--Gu\'eritaud--Kassel--Lee--Marquis \cite{danciger2023convex} that one can arrange for each of the $\rho_n$ to be $P_1$-Anosov in the sense of Labourie \cite{labourie2006anosov} and Guichard--Wienhard \cite{guichard2012anosov}, where $P_1$ denotes the stabilizer of a line in $\mathrm{SL}_n^\pm(\mathbb{R})$. Note that the restriction of $\rho_n$ to any finite-index subgroup of $\Gamma_n$ will then remain $P_1$-Anosov.

Theorem \ref{thm:virtually-Zariski-dense} allows us to conclude the following.

%In particular for any $n\geq5$, the right-angled $n$-gon group embeds as a Zariski-dense subgroup of $\mathrm{SL}_n^{\pm}(\mathbb{Z})$. This implies the following corollary which extends a previous result of Long--Thistlethwaite for $n$ odd \cite{Long_Thistlethwaite_Surface_groups}.

\begin{theorem}\label{thm:surfacesubgroups}
For any $n\geq 3$, there is a Zariski-dense subgroup of $\mathrm{SL}_n(\mathbb{Z})$ isomorphic to the fundamental group of a connected closed orientable surface of genus $\max\{n-3, \> 2\}$ for $n$ odd and genus $\max\{\frac{n-2}{2}, \> 2\}$ for $n$ even. 
\end{theorem}

The existence of Zariski-dense\footnote{We recall that, for $n=3,4$, certain arithmetic lattices in $\mathrm{SO}_{n-1,1}(\mathbb{R})$ yield non-Zariski-dense closed hyperbolic $(n-1)$-manifold groups in $\mathrm{SL}_n(\mathbb{Z})$.} (closed) surface subgroups of $\mathrm{SL}_n(\mathbb{Z})$, regardless of the genus, was previously only known for $n=2k+1$, $k \geq 1$, and for $n=4$. We discuss this history briefly. Note first that $\mathrm{SL}_2(\mathbb{Z})$ is virtually a free group and hence contains no surface subgroups whatsoever. Kac and Vinberg \cite{vinberg1967quasi} constructed Zariski-dense surface subgroups of $\mathrm{SL}_3(\mathbb{Z})$ via hyperbolic triangle groups in a work of great relevance to this paper; see also \cite{long2011zariski}. Long and Thistlethwaite \cite{long2018zariski, Long_Thistlethwaite_Surface_groups} later constructed Zariski-dense surface subgroups in $\mathrm{SL}_4(\mathbb{Z})$ and $\mathrm{SL}_5(\mathbb{Z})$, and then in $\mathrm{SL}_{2k+1}(\mathbb{Z})$ for all $k \geq 1$, though the latter work does not appear to provide any explicit control on the genus as $k$ grows. That $\mathrm{SL}_{2k+1}(\mathbb{Z})$ contains Zariski-dense surface subgroups for all $k \geq 1$ was also announced by Burger--Labourie--Wienhard \cite[Theorem~24]{wienhard2018invitation}. We remark that the surface subgroups mentioned in this paragraph are all Hitchin, whereas any surface subgroup of $\mathrm{SL}_n(\mathbb{R})$ that is of finite index in a linear reflection group in $\mathbb{R}^n$ for $n \geq 4$, and, in particular, the examples in Theorem \ref{thm:surfacesubgroups} for $n \geq 4$, will fail to be Hitchin, as follows from Proposition \ref{prop_dlambda} and \cite[Thm.~1.5]{labourie2006anosov}. For other examples of lattices in split groups admitting thin Hitchin surface subgroups, see \cite{audibert2022zariski, audibert2023zariski}.

It indeed seems reasonable to expect that any irreducible lattice in a semisimple real algebraic group that is not  $\mathrm{PSL}_2(\mathbb{R})$ up to compact groups contains a thin surface subgroup, though the problem of constructing any surface subgroups at all has proved difficult. Nevertheless, the dynamical industry initiated by Kahn and Markovi\'c \cite{kahn2012immersing} for constructing surface subgroups frequently gives rise to Zariski-dense such subgroups; for more on this approach, see \cite{kassel2022groupes} and the references therein.

Apart from surface groups, a wide range of groups are commensurable with irreducible right-angled Coxeter groups of rank $\geq 3$, and hence virtually embed as Zariski-dense subgroups of $\mathrm{SL}_n(\mathbb{Z})$ for all sufficiently large $n$ by Theorem \ref{thm:virtually-Zariski-dense}. These include
\begin{itemize}
\item the fundamental groups of certain closed hyperbolic manifolds of dimensions $3$ and $4$;
\item the fundamental groups of certain complete finite-volume hyperbolic $n$-manifolds for $n \leq 8$ \cite{potyagailo2005rightangled};
\item various exotic Gromov-hyperbolic groups, e.g., certain hyperbolic groups with $3$-sphere boundary that do not (even virtually) embed discretely in $\mathrm{Isom}(\mathbb{H}^4)$, as well as some hyperbolic Poincar\'e duality groups of dimension $4$ whose boundaries are not topological $3$-spheres \cite{przytycki2009flag};
\item irreducible right-angled Artin groups of rank $\geq 2$ \cite{RightAngledCommensurableCoxeter}.
\end{itemize}
However, the minimal dimension $N$ of the embedding guaranteed by Theorem \ref{thm:virtually-Zariski-dense} may be large depending on the group. For example, any irreducible right-angled Coxeter group that is abstractly commensurable with a closed hyperbolic $3$-manifold group has rank $\geq 12$, since, by Mostow rigidity, any such Coxeter group can be realized as a cocompact reflection group acting on $\mathbb{H}^3$ (see \cite[Lemma~5.4]{lee2019ads} and the references therein). By applying Theorem \ref{thm:coxeter} to certain non-right-angled Coxeter groups, we are nevertheless able to produce some new examples of thin groups in lower dimensions as well.

\begin{theorem}\label{thm:hyperbolic_manifolds}
There exist
\begin{itemize}
\item a closed hyperbolic $3$-manifold whose fundamental group virtually embeds Zariski-densely in $\mathrm{SL}_n(\mathbb{Z})$ for all $n \geq 4$ (see Prop.~\ref{prop:application-three-manifold});
\item a complete finite-volume hyperbolic $4$-manifold whose fundamental group virtually embeds Zariski-densely in $\mathrm{SL}_n(\mathbb{Z})$ for all $n \geq 5$ (see Prop.~\ref{prop:application-four-manifold});
\item for every $p \geq 4$, a closed aspherical $p$-manifold whose fundamental group virtually embeds Zariski-densely in $\mathrm{SL}_n(\mathbb{Z})$ for all $n \geq 2p$ (see Prop.~\ref{prop:application-aspherical-manifold}).
\end{itemize}
\end{theorem}

Whether there exists a complete finite-volume hyperbolic $3$-manifold whose fundamental group embeds in $\mathrm{SL}_3(\mathbb{Z})$ is an open question of Long and Reid \cite{long2011small}. Previously, it was known to experts (see \cite{ChoiChoi} and the references therein) that, in the spirit of Kac and Vinberg, there are compact hyperbolic Coxeter simplices in dimensions $3$ and $4$ whose associated Coxeter groups admit Zariski-dense representations as reflection groups into $\mathrm{SL}_4^\pm(\mathbb{Z})$ and $\mathrm{SL}_5^\pm(\mathbb{Z})$, respectively. These representations then give rise to thin closed hyperbolic $3$- and $4$-manifold groups in $\mathrm{SL}_4(\mathbb{Z})$ and $\mathrm{SL}_5(\mathbb{Z})$, respectively. For other manifestations of hyperbolic $d$-manifold groups as thin subgroups of lattices in $\mathrm{SL}_{d+1}(\mathbb{R})$, see \cite{long2014constructing, ballas2015constructing, ballas2020constructing, ballas2020thin}.

\subsection{Zariski-dense witnesses to incoherence of $\mathrm{SL}_n(\mathbb{Z})$}
\noindent
It follows from the simplicity of $\mathrm{SL}_n(\mathbb{R})$ that infinite normal subgroups of Zariski-dense subgroups of $\mathrm{SL}_n(\mathbb{R})$ remain Zariski-dense in $\mathrm{SL}_n(\mathbb{R})$. The phenomenon of ``virtual (algebraic) fibering'' of certain Coxeter groups then yields via the techniques of this paper  some exotic thin subgroups of $\mathrm{SL}_n(\mathbb{Z})$, as illustrated by Theorem \ref{thm:incoherence} below.

A group $\Gamma$ is said to be {\em coherent} if all finitely generated subgroups of $\Gamma$ are finitely presented, and {\em incoherent} otherwise. We will call a finitely generated subgroup of $\Gamma$ that is not finitely presented a {\em witness to incoherence} of $\Gamma$. We focus here on the case $\Gamma = \mathrm{SL}_n(\mathbb{Z})$ for some $n \geq 2$; for a broader perspective on coherence, see Wise's survey \cite{wise2020invitation}, and for a discussion on coherence in the context of lattices in semisimple Lie groups, see the introduction of \cite{kapovich2013noncoherence}.

That $\mathrm{SL}_2(\mathbb{Z})$ is coherent can be deduced for instance from the fact that finitely generated Fuchsian groups are geometrically finite. Whether $\mathrm{SL}_3(\mathbb{Z})$ is coherent is a well-known question of Serre \cite[Prob.~F14]{wall1979homological} and remains open; indeed, it appears that all known examples of thin subgroups of $\mathrm{SL}_3(\mathbb{Z})$ are abstractly commensurable to Fuchsian groups. 
Since $F_2 \times F_2$ is incoherent \cite[Example~9.22]{wise2020invitation}, where~$F_2$ denotes the free group of rank two, the existence of an $F_2 \times F_2$ subgroup of $\mathrm{SL}_4(\mathbb{Z})$ precludes coherence of $\mathrm{SL}_n(\mathbb{Z})$ for $n \geq 4$. Note however that $F_2 \times F_2$ cannot embed Zariski-densely in a simple Lie group, and in particular, no witness to incoherence of $\mathrm{SL}_n(\mathbb{Z})$ that is contained within an $F_2 \times F_2$ subgroup of $\mathrm{SL}_n(\mathbb{Z})$ will be Zariski-dense in $\mathrm{SL}_n(\mathbb{R})$. This led Stover \cite{stover2019coherence} to ask whether there are Zariski-dense witnesses to incoherence of $\mathrm{SL}_n(\mathbb{Z})$ for $n \geq 4$. 

It was pointed out to the second author by Konstantinos Tsouvalas that there are witnesses to incoherence of $\mathrm{SL}_n(\mathbb{Z})$ for $n \geq 5$ that decompose as $\Delta * F$, where $\Delta$ is a witness to incoherence of $\mathrm{SO}_{4,1}(\mathbb{Z})$ contained within an infinite-index geometrically finite subgroup of the latter (see \cite{MR2484708}) and $F$ is a Zariski-dense free subgroup of $\mathrm{SL}_n(\mathbb{Z})$. For $n \geq 6$, we exhibit in \S \ref{sec:incoherence} some Zariski-dense witnesses to incoherence of $\mathrm{SL}_n(\mathbb{Z})$ of a different nature.

\begin{theorem}\label{thm:incoherence}
	For each $n \geq 6$, there is a Zariski-dense \emph{one-ended} finitely generated subgroup of $\mathrm{SL}_n(\mathbb{Z})$ that is not finitely presented.
\end{theorem}

Previously, the second author described in \cite{douba2024novel} a certain Zariski-dense witness to incoherence of $\mathrm{SL}_5(\mathbb{Z})$ constructed, as in the proof of Theorem \ref{thm:incoherence}, via a virtual fiber subgroup of a linear reflection group. It follows from forthcoming work of Fisher--Italiano--Kielak \cite{fisher2025virtual}
that the example in~\cite{douba2024novel} is also one-ended. By the Scott core theorem \cite{scott1973finitely}, no subgroup of a linear reflection in $\mathbb{R}^4$ will serve as a witness to incoherence of $\mathrm{SL}_4(\mathbb{Z})$, since such a reflection group preserves and acts properly on a domain in $\mathbb{P}(\mathbb{R}^4)$. We remain unaware of a Zariski-dense witness to incoherence of $\mathrm{SL}_4(\mathbb{Z})$, one-ended or otherwise.

\subsection{Relation to previous work of Benoist}
\noindent
A representation $\rho$ as in the setting of Theorem \ref{thm:coxeter} preserves a properly convex domain in the projective space $\mathbb{P}(V)$. In \cite{auto_convex_benoist}, Benoist described the Zariski-closures of irreducible representations (of arbitrary groups) preserving such a domain in~$\mathbb{P}(V)$. Our proof of Theorem \ref{thm:coxeter} however does not rely on Benoist's description, nor are we aware of a more straightforward proof that does so.

\subsection*{Organization of the paper} In \S\ref{sec:Vinberg-theory}, we review Vinberg's theory of linear reflection groups. Sections \ref{sec:reflection-groups-contain-dlambda}, \ref{sec:lie-groups-containing-dlambda}, and \ref{sec:proof-of-main} are devoted to the proof of Theorem \ref{thm:coxeter}. In \S\ref{sec:thin-reflection-groups}, we give the first applications of Theorem \ref{thm:coxeter}, and in particular, prove Theorems \ref{thm:virtually-Zariski-dense} and \ref{thm:surfacesubgroups}. We present in \S\ref{sec:thin-hyperbolic-manifold-groups} some constructions of thin hyperbolic manifold groups via linear reflection groups (Theorem~\ref{thm:hyperbolic_manifolds}). Finally, we prove Theorem \ref{thm:incoherence} in \S\ref{sec:incoherence}.

%%%%%%%%%%%%%%%%%%%%%%%%%
\subsection*{Acknowledgements}

The second author thanks Sam P. Fisher, Giovanni Italiano, and Dawid Kielak for sharing their work on one-endedness of fiber subgroups. The second author became involved in this project while on a visit to the Jagiellonian University, Krak\'ow, in early 2024, and wishes to thank Miko{\l}aj Fr\k{a}czyk and the entire Dioscuri Centre ``Random Walks in Geometry and Topology'' for their hospitality. The third and fourth authors are especially grateful to Martin D. Bobb, since they started thinking about this topic with him several years ago. We also thank Fanny Kassel, Robbie Lyman, and Beatrice Pozzetti for helpful conversations.

This work was supported by a grant from the Fondation Math\'ematique Jacques Hadamard. J.A. acknowledges the support of the Max-Planck Institute for Mathematics (Bonn), and of the Max-Planck Institute for Mathematics in the Sciences (Leipzig). S.D. was supported by the Huawei Young Talents Program. G.L. was supported by the European Research Council under ERC-Consolidator Grant 614733 and by the National Research Foundation of Korea (NRF) grant funded by the Korea government (MSIT) (RS-2023-00252171). L.M. acknowledges support by the Centre Henri Lebesgue (ANR-11-LABX-0020 LEBESGUE),  ANR G\'eom\'etries de Hilbert sur tout corps valu\'e (ANR-23-CE40-0012) and  ANR Groupes Op\'erant sur des FRactales (ANR-22-CE40-0004). 

\section{Vinberg's theory of reflection groups} \label{sec:Vinberg-theory}

\subsection{Coxeter groups}

A \emph{Coxeter matrix} $M$ on a finite set $S$ is a symmetric $S \times S$ matrix $M=(m_{st})_{s,t \in S}$ with entries $m_{st} \in \{ 1, 2, \dotsc, \infty \} $ with diagonal entries $m_{ss}=1$ and off-diagonal entries $m_{st} \neq 1$. To a Coxeter matrix $M$ is associated a \emph{Coxeter group} $W_S$: the group presented by the set of generators $S$ and the relations $(st)^{m_{st}}=1$ for each $(s,t) \in S \times S$ with $m_{st} \neq \infty$. To be more precise, we will take the datum of a Coxeter group $W_S$ to include the generating set~$S$. The cardinality of $S$ is called the \emph{rank} of the Coxeter group $W_S$ and is denoted by $\textrm{rank}(W_S)$. We say $W_S$ is {\em right-angled} if $m_{st} \in \{2, \infty\}$ for all $s,t \in S$.

\smallskip

The combinatorics of a Coxeter group $W_S$ are encoded in a labeled simplicial graph called the \emph{Coxeter diagram} $\mathcal{G}_{W_S}$ of $W_S$, defined as follows. The vertex set of $\mathcal{G}_{W_S}$ is $S$. Two vertices $s,t \in S$ are joined by an edge $\overline{st}$ if and only if $m_{st} \in \{ 3, 4, \dotsc, \infty \} $, and the label of the edge $\overline{st}$ is~$m_{st}$. It is customary to omit the label of the edge $\overline{st}$ if $m_{st} = 3$. 

\smallskip

For any subset $T$ of $S$, the $T \times T$ submatrix of $M$ is a Coxeter matrix on $T$. We may identify~$W_{T}$ with the subgroup of $W_S$ generated by $T$ (see \cite[Chap.~4, \S1.8]{Bourbaki_LiegroupsandLiealgebras46}). Such a subgroup is called a \emph{standard subgroup} of $W_S$.

\smallskip

An element $\gamma \in W_S$ is a {\em reflection} if $\gamma$ is conjugate within $W_S$ to an element of $S$. A subgroup $\Gamma < W_S$ is a {\em reflection subgroup} if $\Gamma$ is generated by reflections. Such a subgroup $\Gamma$ is finitely generated if and only if $\Gamma$ is generated by finitely many reflections, in which case $\Gamma$ is naturally isomorphic to a Coxeter group \cite{deodhar1989note, dyer1990reflection}. Note that standard subgroups of $W_S$ are examples of finitely generated reflection subgroups. Note also that if $W_S$ is right-angled, then the same is true of any of its finitely generated reflection subgroups.

A Coxeter group $W_S$ is said to be \emph{irreducible} if $\mathcal{G}_{W_S}$ is connected. Otherwise, the connected components of the Coxeter diagram $\mathcal{G}_{W_S}$ are Coxeter diagrams of the form $(\mathcal{G}_{W_{S_i}})_i$, where the $(S_i)_i$ form a partition of $S$. The subgroups $(W_{S_i})_i$ are called the \emph{components} of $W_S$. A Coxeter group $W_S$ is \emph{spherical} (resp. \emph{affine}) if each component of $W_S$ is finite (resp. infinite and virtually abelian). Note that every irreducible Coxeter group $W_S$ is spherical, affine, or \emph{large}, i.e., has a finite index subgroup with a non-abelian free quotient \cite[Cor.\ 2]{Margulis2000}.

\subsection{Cartan matrices}

A \emph{Cartan matrix} on a set $S$ is a matrix $\Cart=(\Cart_{st})_{s,t \in S}$ satisfying the conditions

\begin{itemize}
	\item $\forall s \in S, \quad \Cart_{ss} = 2$;
	\item $\forall s \neq t \in S, \quad \Cart_{st} \leqslant 0$;
	\item $\forall s,t \in S, \quad \Cart_{st} = 0 \quad \Leftrightarrow \quad \Cart_{ts} = 0$.
\end{itemize}

A Cartan matrix $\Cart$ is \emph{decomposable} if there exists a nontrivial partition of $S$ such that~$\Cart$ written with respect to this partition is block-diagonal. Otherwise, $\Cart$ is \emph{indecomposable}. Two Cartan matrices are \emph{equivalent} if they are conjugate via a diagonal matrix all of whose diagonal entries are positive (this is easily seen to be an equivalence relation). A Cartan matrix is \emph{symmetrizable} if it is equivalent to a symmetric matrix.

A Cartan matrix $\Cart_S$ and a Coxeter group $W_S$ are \emph{compatible} if
\begin{enumerate}
	\item $\forall s,t \in S,\,  m_{st} = 2 \quad\,\,\,\,\,\, \Leftrightarrow \quad \Cart_{st} = 0$;
	\item $\forall s,t \in S$, $m_{st} < \infty \quad \Rightarrow \quad \Cart_{st} \Cart_{ts} = 4 \cos^2 ( \nicefrac{\pi}{m_{st}} )$;
	\item $\forall s,t \in S$, $m_{st} = \infty \quad \Rightarrow \quad \Cart_{st} \Cart_{ts} \geqslant 4$.
\end{enumerate}

\begin{remark}\label{rem:tits_matrix}
	For an arbitrary Coxeter group $W_S$, the matrix given by $\Cart_{st} = - 2\cos(\nicefrac{\pi}{m_{st}})$ is compatible with $W_S$. This matrix is called the \emph{Tits matrix} of $W_S$.
\end{remark}

If $\Cart$ is an indecomposable Cartan matrix then the matrix $2\mathrm{Id} - \Cart$ is an irreducible Perron--Frobenius matrix\footnote{An {\em irreducible Perron--Frobenius matrix} is a nonnegative matrix $A = (A_{st})_{s,t \in S}$ such that for every $s,t \in S$, there exists $k > 0$ such that the $(s,t)$ entry of $A^k$ is positive.}, since the directed graph associated to  $2\mathrm{Id} - \Cart$ is clearly symmetric (by the definition of a Cartan matrix) and connected (since $\Cart$ is indecomposable). Hence, the spectral radius of $2\mathrm{Id} - \Cart$ is a simple eigenvalue $\rho>0$ of $2\mathrm{Id} - \Cart$, by Perron--Frobenius's theorem. It follows that $2-\rho\in (-\infty,2)$ is the eigenvalue with the smallest real part of $\Cart$. If $2-\rho > 0$ (resp., $=0$, $<0$) then $\Cart$ is said to be of \emph{positive type} (resp., \emph{zero type}, \emph{negative type}).

\begin{proposition}\emph{\cite[Lem.~13, Prop~21-23]{the_bible}\label{prop:types_correlation}}
	Let $W_S$ be an irreducible Coxeter group and~$\Cart$ a compatible Cartan matrix. Then:
	\begin{itemize}
		\item $W_S$ is spherical if and only if $\Cart$ is of positive type.
		
		\item If $\Cart$ is of type zero, then $W_S$ is affine.
		
		\item If $\Cart$ is of negative type and $W_S$ is affine, then $W_S$ is of type $\widetilde{A}_{|S|-1}$.
		
		\item If $W_S$ is large, then $\Cart$ is of negative type.
	\end{itemize}
\end{proposition}

\subsection{Reflection groups}

A \emph{reflection} $\sigma$ of $V$ is an endomorphism of $V$ of order 2 which fixes a hyperplane of $V$ pointwise. Hence there exists a vector $v \in V$ and a linear form $\alpha \in V^{*}$ with $\alpha(v)=2$ such that $ \sigma = \mathrm{\mathrm{Id}} -\alpha\otimes v,$ i.e., such that $\sigma(x) = x - \alpha(x)v$ for all $x \in V$.
Given a reflection~$\sigma$, the pair $(\alpha,v) \in V^\ast \times V$ is not unique. Any other such pair is of the form $(\lambda \alpha, \lambda^{-1} v)$ with $\lambda \neq 0$.

\begin{definition}\label{def:refl_gp}
	Let $W_S$ be a Coxeter group.	A representation $\rho:W_S \to \mathrm{GL} (V)$ is \emph{a representation of $W_S$ as a reflection group} if for every $s\in S$ there exist $v_s\in V$ and $\alpha_s\in V^*$ such that:
	\begin{enumerate}
		\item\label{item:rho(s)} for every $s \in S$, $\rho(s) = \Id - \alpha_s \otimes v_s$;
		
		\item $\Cart =(\alpha_s(v_t))_{s,t \in S}$ is a Cartan matrix;
		
		\item the Cartan matrix $\Cart$ and the Coxeter group $W_S$ are compatible;
		
	 	\item the convex cone $\widetilde{\Delta} := \bigcap_{s \in S} \{ \alpha_s \leqslant 0 \}$ has nonempty interior.
	\end{enumerate}
\end{definition}

	The family $(\alpha_s,v_s)_{s \in S}$ in the previous definition is not unique. However, if $W_S$ is irreducible, the equivalence class of $\Cart$ depends only on $\rho$, and we refer to this equivalence class as the {\em Cartan matrix of $\rho$}.

\begin{remark}\label{rem:restriction_to_standard_subgroup}
	By definition, if $\rho: W_S \to \GL (V)$ is a representation of $W_S$ as a reflection group with Cartan matrix $\Cart$ and $T \subset S$ then the restriction $\rho\bigr|_{W_T} : W_T \to \GL (V)$ is a representation of $W_T$ as a reflection group with Cartan matrix $\Cart_T := (\Cart_{st})_{s,t \in T}$.
\end{remark}

\begin{example}\label{eg:representation_from_cartan}
Let $W_S$ be a Coxeter group. To any Cartan matrix $\Cart$ compatible with $W_S$, one can associate a representation $\rho_\Cart : W_S \rightarrow \mathrm{GL}_{|S|}(\mathbb{R})$ sending each $s \in S$ to the reflection
\[
v \mapsto v - (v^T\Cart e_s)e_s \qquad
\]
of $\mathbb{R}^{|S|}$. If $W_S$ is irreducible and large, then $\rho_\Cart$ is a representation of $W_S$ as a reflection group with Cartan matrix $\Cart$; see \cite[Rem.~3.14(i)]{danciger2023convex}. If $\Cart$ is the Tits matrix of $W_S$ (see Remark \ref{rem:tits_matrix}), then $\rho_\Cart$ is called the {\em geometric representation} of $W_S$.
\end{example}

\begin{theorem}[Tits, Vinberg]\label{thm:tits-vinberg1}
Let $W_S$ be a Coxeter group and $\rho:W_S\to \mathrm{GL} (V)$ a representation of $W_S$ as a reflection group. Then:
\begin{enumerate}
		\item\label{item:faithful_discrete} the representation $\rho$ is faithful and $\rho(W_S)$ is discrete in $\mathrm{GL} (V)$;
		
		\item\label{item:convex} the union of the convex cones $\rho(\gamma) (\widetilde{\Delta})$ as $\gamma$ varies within $W_S$ is a convex cone $\Cc_{\mathrm{TV}}$ of $V$;
		
		\item\label{item:action_proper} the action of $\rho (W_S)$ on the interior $\Omega_{\mathrm{TV}}$ of $\Pb(\Cc_{\mathrm{TV}})$ is proper. The domain $\Omega_{\mathrm{TV}}$ is called the \emph{Tits--Vinberg domain} of the reflection group $\rho (W_S)$.
\end{enumerate}

\end{theorem}

\begin{proof}
This comes from Bourbaki \cite[Chap.~V.~\S4.4-6]{Bourbaki_LiegroupsandLiealgebras46} in a special case, and Vinberg \cite[Thm.~2]{the_bible} in the general case. See \cite[Lec.~1]{fivelectures} for a self-contained short proof. 
\end{proof}

Given a representation $\rho: W_S \rightarrow \mathrm{GL}(V)$ of $W_S$ as a reflection group, define $$V_v=\mathrm{Span}(v_s\mid s\in S)\ \textrm{and}\ V_{\alpha}=\bigcap\limits_{s\in S}\ker(\alpha_s).$$ These two subspaces of $V$ are well defined and preserved by $\rho(W_S)$. 

We denote by $V_{\Cb}$ the complexification of the real vector space $V$. A representation $V$ of a group is \emph{absolutely irreducible} if $V_\mathbb{C}$ does not admit any nontrivial proper invariant $\mathbb{C}$-subspaces. 

In the following theorem, we gather several well-known results about representations of Coxeter groups as reflection groups.

\begin{theorem}[Folklore]\label{thm:tits-vinberg2}	
	Let $W_S$ be an irreducible Coxeter group and $\rho:W_S\to \mathrm{GL} (V)$ a representation of $W_S$ as a reflection group. Let $\Cart$ be the Cartan matrix of $\rho$. Then: 
	\begin{enumerate}
		
		\item\label{item:irreducible_rep} The representation $\rho$ is irreducible if and only if $V_v=V$ and $V_{\alpha}=\{0\}$. In this case, $\dim (V) = \mathrm{rank} (\Cart)$.
		
		\item\label{item:bilinear_form} Assume $\rho$ is irreducible. Then $\rho$ preserves a nonzero (hence nondegenerate) symmetric bilinear form on $V$ if and only if $\Cart$ is symmetrizable.
		
		\item\label{item:invariant_pcc} If $W_S$ is large and $V_\alpha = \{ 0\}$, then $\Omega_{\mathrm{TV}}$ is a properly convex open subset of $\Pb (V)$.
		
		\item\label{item:strongly_irreducible} If $W_S$ is large and $\rho$ is irreducible, then the restriction of $\rho$ to any finite-index subgroup of $W_S$ is absolutely irreducible.
	\end{enumerate} 
\end{theorem}

%
%(Bourbaki , Vinberg \cite{the_bible}, \cite[Proposition 3.23]{danciger2023convex}) \label{thm:tits-vinberg}
%
%1,2,3, Vinberg \cite[Thm.~2]{the_bible}
%4, Vinberg \cite[Prop.~19]{the_bible}, see also DGKLM
%5, Vinberg \cite[Thm.~6]{the_bible}
%6, Vinberg \cite[Lem.~15]{the_bible}
%7, \cite[Proposition 3.23]{danciger2023convex}

\begin{proof}
	Item~\eqref{item:irreducible_rep} comes from Vinberg \cite[Prop.~18-19]{the_bible} or \cite[Chap.~V.~\S4.7]{Bourbaki_LiegroupsandLiealgebras46} in a special case. See also \cite[Proposition 3.23]{danciger2023convex}.
	Item~\eqref{item:bilinear_form} comes from \cite[Thm.~6]{the_bible}. %The reader can also look at Lemma~\ref{lemma_symmetrizable}.
	For Item~\eqref{item:invariant_pcc}, Proposition~\ref{prop:types_correlation} shows that $\Cart$ is of negative type and \cite[Lem.~15]{the_bible} shows that~$\O_{\mathrm{TV}}$ is properly convex. Item~\eqref{item:strongly_irreducible} comes from \cite[Proposition 3.23]{danciger2023convex}, see also \cite[Lem.~1]{de_la_harpe}. In fact, the proof is written for real vector spaces but holds for complex vector spaces.
\end{proof}

\subsection{Reducing to an irreducible representation}
\label{subsection:Invariantsubspace}
A representation $\rho: W_S \to \GL (V)$ of~$W_S$ as a reflection group gives rise to a representation $\rho_v^\alpha$ of $W_S$ on $V_v/(V_\alpha\cap V_v)$. If $W_S$ is irreducible and large, then the representation $\rho_v^\alpha$ is an irreducible representation of $W_S$ as a reflection group, shares the same Cartan matrix $\Cart$ as $\rho$, and has dimension $\mathrm{rank}(\Cart)$; see \cite[Sec.~3.9]{danciger2023convex}.

For $\lambda\in\Cb$, define $$D(\lambda):=\textrm{diag}(\lambda,1,\dots,1,\lambda^{-1}).$$

\begin{lemma}\label{lem:restrict_to_irreducible_rep}
Let $W_S$ be an irreducible and large Coxeter group and $\rho: W_S \to \GL (V)$ a representation of $W_S$ as a reflection group. For any $\lambda > 1$, for any $\sigma,\tau$ reflections of $W_S$, $\rho (\sigma \tau)$ is conjugate to 
the matrix $D(\lambda)$ if and only if $\rho_v^\alpha (\sigma \tau)$ is conjugate to the matrix $D(\lambda)$.
\end{lemma}

\begin{proof}
Let $\sigma,\tau$ be two reflections of $W_S$, $\alpha,\beta$ two linear forms on $V$ and $v,w$ vectors of $V$ such that $\alpha (v) = \beta (w) = 2$, $\rho(\sigma) = \Id - \alpha \otimes v$ and $\rho(\tau) = \Id - \beta \otimes w$. The product $\rho( \sigma \tau)$ is conjugate to a matrix $D(\lambda)$, with $\lambda > 1$ if and only if $\alpha (w) \beta (v) > 4$ (\cite[Proof of Prop.~6]{the_bible} or \cite[Lem.~1.2]{fivelectures}).
Denote by $\overline{\alpha},\overline{\beta}$ the linear forms induced by $\alpha,\beta$ on $V_v/(V_\alpha \cap V_v)$, and by $\overline{v},\overline{w}$ the projections of the vectors $v,w$ onto $V_v/(V_\alpha \cap V_v)$, respectively. The statement now follows from the fact that $\alpha (w) = \overline{\alpha} (\overline{w})$ and $\beta (v) = \overline{\beta} (\overline{v})$.
\end{proof}

\section{Large reflection groups always contain a matrix $D (\lambda)$}\label{sec:reflection-groups-contain-dlambda}

 The goal of this section is to show the following.

\begin{proposition}
	\label{prop_dlambda}
	Let $W$ be an irreducible large Coxeter group and $\rho:W\to\GL(V)$ a representation of $W$ as a reflection group. Then $\Gamma : = \rho(W)$ contains an element which is conjugate within $\GL(V)$ to $D({\lambda})$ for some real number $\lambda > 1$.
\end{proposition}

Theorem \ref{thm:coxeter} will be a consequence of Proposition \ref{prop_dlambda} and Theorem \ref{thm:rep_containing_d_lambda}. The latter will be proved in \S\ref{sec:lie-groups-containing-dlambda}.

\subsection{Large Coxeter groups contain standard quasi-Lannér Coxeter groups}

A large irreducible Coxeter group $W_T$ is \emph{quasi-Lannér} if for every $t \in T$, the standard subgroup $W_{T \smallsetminus t}$ is spherical or irreducible affine. Quasi-Lannér Coxeter groups were classified by Koszul  \cite{LectHypCoxGrKoszul} and Chein \cite{chein}. The geometric representation of a quasi-Lannér Coxeter group of rank~$N$ preserves a unique (up to a positive scalar) Lorentzian bilinear form $B$ on $\R^N$, and hence preserves a unique ellipsoid $\Ec$ of $\Pb (\R^N)$. This ellipsoid, endowed with the Hilbert metric, is isometric to the real hyperbolic space of dimension $N-1$. Write the image of $t\in T$ under the geometric representation as $\mathrm{Id}-\alpha_t\otimes v_t$,  where $\alpha_t\in V^*$ and $v_t\in V$. The polytope $\Delta = \Pb (\{ \alpha_t (x) \leqslant 0 \})$ is contained in $\overline{\Ec}$ and is a finite-volume simplex. Conversely, if $\Delta$ is a finite-volume simplex in the hyperbolic space $\Hb^{d}$, $d \geqslant 2$, all whose dihedral angles are submultiples of $\pi$, then the group generated by the hyperbolic reflections across the facets of $\Delta$ is a quasi-Lannér Coxeter group.

\begin{lemma}[See also de la Harpe \cite{de_la_harpe} and Edgar \cite{Edgar_quasi_lanner}]\label{lem:sub_quasi_lanner}
	Let $W_S$ be a large Coxeter group. There exists $T \subset S$ such that $W_T$ is quasi-Lannér.
\end{lemma}

\begin{proof}
	Consider $\mathcal{T} = \{ T \subset S \,|\, W_T \textrm{ is irreducible large}\}$ and let $T$ be a minimal element of $\mathcal{T}$. We claim that $W_T$ is quasi-Lannér. Pick $t \in T$. By minimality of $T$, any irreducible factor of $W_{T \smallsetminus t}$ is spherical or affine. Suppose that $W_{T \smallsetminus t}$ is not spherical. Then $W_{T \smallsetminus t}$ has at least one affine irreducible factor $W_U$. But $W_T$ is irreducible, so $W_{U \cup t}$ must be irreducible. By the classification of affine Coxeter groups, $W_{U \cup t}$ is not affine and is hence large, and so, by minimality of $T$, we see that $T = U \cup \{ t\}$. Hence, $W_{T \smallsetminus t}$ is irreducible affine. We conclude that~$W_T$ is quasi-Lannér.
\end{proof}

\subsection{Proof of Proposition~\ref{prop_dlambda} in the presence of an invariant round domain}

A convex domain, i.e., a properly convex open subset, $\O \subset \Pb (V)$ is \emph{round} if $\O$ is strictly convex with~$\Cc^1$ boundary. In this section, we prove Proposition \ref{prop_dlambda} in the case where $\rho$ is irreducible and preserves a round convex domain. This step will be needed later for the proof of the general case.

Let $\O \subset \Pb (V)$ be a round convex domain and let $\sigma=\Id - \alpha \otimes v$ be a linear reflection of $V$ which preserves $\Omega$. We say that $\sigma$ is a \emph{reflection of $\Omega$} if $[\ker (\alpha)] \cap \Omega \neq \varnothing$. In this case, $H_\sigma$ denotes the subset $[\ker (\alpha)] \cap \overline{\Omega}$ and is called \emph{the hyperplane of reflection of $\sigma$ (in $\Omega$)}. The \emph{polar}~$p_\sigma$ of the reflection $\sigma$ is the point $p_\sigma = [v] \in \Pb (V)$.

\begin{remark}\label{rem:product_two_refl_with_disjoint_hyperplanes}
	If $\sigma,\tau$ are two reflections of $\Omega$ such that $H_\sigma \cap H_\tau = \varnothing$, then $\sigma \tau$ is conjugate to~$D(\lambda)$ for some $\lambda > 1$.
\end{remark}

%The following Lemma is a sub-case of Proposition~\ref{prop_dlambda} that will be useful to prove Proposition~\ref{prop_dlambda}.

\begin{lemma}\label{lem:round_case}
	Let $\Gamma < \GL (V)$ be a discrete subgroup that acts irreducibly on $V$ and preserves a round convex domain $\Omega \subset \mathbb{P}(V)$. If $\Gamma$ contains a reflection of $\Omega$, then $\Gamma$ contains two reflections $\sigma,\tau \in \Gamma$ of $\Omega$ such that $\sigma \tau$ is conjugate to $D(\lambda)$ for some $\lambda > 1$.
\end{lemma}

For $a \in \partial \Omega$, denote by $T_a \partial \Omega$ the tangent hyperplane to $\partial \O$ at $a$. To show Lemma~\ref{lem:round_case}, we will need the two following lemmas.

\begin{lemma}\label{lem:polar}
	Let $\Omega \subset \Pb (V)$ be a round convex domain and $\sigma$ be a reflection of $\O$. If $a \in H_\sigma \cap \partial \Omega$, then $p_\sigma \in T_a \partial \Omega$. 
\end{lemma}

\begin{proof}
	Since $\Omega$ is round, we have that $T_a \partial \Omega$ is $\sigma$-invariant. Furthermore, $T_a \partial \Omega$ is not contained in the span of $H_\sigma$, hence $T_a \partial \Omega$ must contain the polar of $\sigma$.
\end{proof}

\begin{lemma}\label{lem:polar_limit}
	Let $\O \subset \Pb (V)$ be a round convex domain and $(\sigma_n)_{n \in \N}$ a sequence of reflections of~$\O$.
	Suppose that the polars of the $\sigma_n$ converge to $a \in \partial \Omega$. Then the closed subsets $H_{\sigma_n}$ converge to $a$ in the Hausdorff topology.
\end{lemma}

\begin{proof}
	Clearly, the span $U_n$ of $H_{\sigma_n}$ converges to $T_a \partial \Omega$. This implies that $H_{\sigma_n}$ converges to $T_a \partial \Omega \cap \overline{\Omega}$, which is equal to $\{a\}$ since $\O$ is round.
\end{proof}

\begin{proof}[Proof of Lemma~\ref{lem:round_case}]
	Since $\G$ acts irreducibly on $V$ and preserves a properly convex domain in $\mathbb{P}(V)$, we have that $\G$ contains a proximal element by \cite[Prop.~3.1]{auto_convex_benoist}. Hence, the proximal limit set $\Lambda_\Gamma$ of $\Gamma$, which is the closure in $\Pb(V)$ of the set of attracting fixed points of proximal elements of $\Gamma$, is nonempty. 
Furthermore, every nonempty closed subset of $\Pb(V)$ which is $\Gamma$-invariant contains $\Lambda_{\Gamma}$ \cite[Lemma 2.5]{auto_convex_benoist}. In particular, for every $x\in\Pb(V)$, $\Lambda_{\Gamma}\subset\overline{\Gamma\cdot x}.$
	
	\smallskip
	
	%Since $\Omega$ is strictly-convex, $\Lambda_{\Gamma}=\overline{\G \cdot x} \smallsetminus \G \cdot x$ for any $x \in \Omega$.
	%The limit set is a perfect compact subset of $\partial \O$ since $\Omega$ is strictly-convex. In particular, it is infinite. Moreover, by definition of proximality, if $y \notin \overline{\O}$ then  $\overline{\G \cdot y}\supset \Lambda_\Gamma$.
	
	\smallskip
	
	Since $\Lambda_{\Gamma}$ is invariant under $\Gamma$, we must have that $\Lambda_{\Gamma}$ contains at least two points. Let $a,b$ be two distinct points of $\Lambda_\Gamma$ and let $p$ be the polar of a reflection $\sigma\in\Gamma$ of $\Omega$. By the previous paragraph, there exist two sequences $(\gamma_n)_n$ and $(\delta_n)_n$ in $\Gamma$ such that $\gamma_n (p)$ converges to $a$ and $\delta_n (p)$ converges to $b$. Hence, the hyperplanes of the reflections $\gamma_n \sigma \gamma_n^{-1}$ (resp., $\delta_n \sigma \delta_n^{-1}$) converge to $a$ (resp., $b$) in the Hausdorff topology by Lemma~\ref{lem:polar_limit}. For $n$ large enough, the hyperplanes $H_{\gamma_n \sigma \gamma_n^{-1}}$ and $H_{\delta_n \sigma \delta_n^{-1}}$ are disjoint, so by Remark~\ref{rem:product_two_refl_with_disjoint_hyperplanes}, the product $(\gamma_n \sigma \gamma_n^{-1})(\delta_n \sigma \delta_n^{-1})$ is conjugate to $D(\lambda)$ for some $\lambda > 1$.
\end{proof}

\subsection{Proof of Proposition~\ref{prop_dlambda} in the general case}

We will need another lemma.

\begin{lemma}\label{lem:irr=>dimV=N}
	Let $W_S$ be a quasi-Lannér Coxeter group and $\rho:W_S\to\GL(V)$ a representation of $W_S$ as a reflection group. If $W_S$ is of rank $N$ and $\rho$ is irreducible, then $\dim (V) = N$.
\end{lemma}

\begin{proof}
Let $r$ be the rank of the Cartan matrix $\Cart$ of $\rho$. By Theorem~\ref{thm:tits-vinberg2}.\eqref{item:irreducible_rep}, $\dim (V) = r$, and so we have to show that $r=N$. It follows from \cite[Lem.~18]{the_bible} that if there exists $t \in S$ such that $W_{S \smallsetminus t}$ is spherical, then $r=N$. Similarly, the proof of \cite[Prop.~26]{the_bible} shows that if there exists $t \in S$ such that $W_{S \smallsetminus t}$ is irreducible affine and $\Cart_{S \smallsetminus t}$ is of zero type, then $r=N$.

\smallskip

Since $W_S$ is quasi-Lannér, for every $t \in S$, we have that $W_{S \smallsetminus t}$ is spherical or irreducible affine. By Proposition \ref{prop:types_correlation} and the previous paragraph, we have to exclude the following: assume there exists $t \in S$ such that $W_{S \smallsetminus t}$ is affine, $\Cart_{S \smallsetminus t}$ is of negative type, and $r < N$. In this case, $\dim( \Omega_{\mathrm{TV}}) = r-1 \leqslant N-2$. First, if $N=3$, then $W_S$ is a hyperbolic triangle group which acts properly discontinuously on the $1$-dimensional properly convex domain $\Omega_{\mathrm{TV}}$. This is absurd.

Assume $N \geqslant 4$. Proposition \ref{prop:types_correlation} shows that $W_{S \smallsetminus t} \simeq \widetilde{A}_{N-2}$.  The convex domain $\Omega_{\mathrm{TV}}$ is evidently preserved by $\rho (W_{S \smallsetminus t})$. However the representations of dimension $\leqslant N-1$ of $\widetilde{A}_{N-2}$ as a reflection group with negative-type Cartan matrix are irreducible\footnote{since such a Cartan matrix is invertible by a straightforward determinant computation.} and each preserve a unique convex domain which is a simplex (see \cite[Lem.~3.22.(c)]{danciger2023convex}). Hence, $\Omega_{\mathrm{TV}}$ must be a simplex, so that the lines given by the vertices of that simplex are permuted by $W_S$. It follows that the restriction of $\rho$ to some finite-index subgroup of $W_S$ is not irreducible, which is absurd by Theorem~\ref{thm:tits-vinberg2}.\eqref{item:strongly_irreducible}.
\end{proof}

\begin{proof}[Proof of Proposition \ref{prop_dlambda}]
		By Proposition~\ref{lem:sub_quasi_lanner}, there exists a standard subgroup $W_S$ of $W$ which is quasi-Lannér. By Remark~\ref{rem:restriction_to_standard_subgroup}, the restriction of a representation as a reflection group to a standard subgroup is again a representation as a reflection group, so that we may assume that $W$ is quasi-Lannér. By Lemma~\ref{lem:restrict_to_irreducible_rep}, we may assume that $\rho$ is irreducible. Denote by $\Cart$ the Cartan matrix of $\rho$.
	
	\smallskip
	 Let $N$ be the rank of $W$ and $\Delta = \Pb (\{ \bigcap_{s \in S} \alpha_s \leqslant 0\})$. Lemma~\ref{lem:irr=>dimV=N} shows that $\dim (V) = N$. The polytope $\Delta$ is of dimension $N-1$ and has $N$ facets, so $\Delta$ is a simplex.
	 
 	 There are two cases to distinguish:
	 \begin{enumerate}
	 	\item\label{item:round} for every affine subset $T \subset S$, the Cartan matrix $\Cart_T$ is of zero type;
	 	
	 	\item\label{item:not_round} there exists an affine subset $T \subset S$ such that $\Cart_T$ is of negative type. In that case, $W_T \simeq \widetilde{A}_{N-1}$ by Proposition \ref{prop:types_correlation}.
	 \end{enumerate}
	 
	Assume case \eqref{item:round}. We claim that $\Delta$ is a 2-perfect Coxeter polytope in the sense of \cite{cox_in_hil}. Fix a vertex $x$ of $\Delta$. Among the $N$ facets of $\Delta$, precisely one does not contain $x$. Denote by $t\in S$ the corresponding generator. Since $\rho(W_{S\setminus t})$ preserves $x$, the representation $\rho$ induces a representation of $W_{S\setminus t}$ on $V/x$ as a reflection group. Denote by $\Omega_x$ the associated Tits--Vinberg domain. Since $W_{S \smallsetminus t}$ is spherical or irreducible affine (and the link of $x$ in $\Delta$ is a simplex), the action of $W_{S \smallsetminus t}$ on $\Omega_x$ is cocompact \cite[Thm.~2.(3)]{the_bible}. This shows that $\Delta$ is a 2-perfect Coxeter polytope. Note that $W_S$ is hyperbolic relative to its irreducible affine subgroups since $W_S$ can be realized as the reflection group associated to a finite-volume hyperbolic simplex. Moreover, the subsets $T \subset S$ such that $\Cart_T$ is of zero type are exactly the affine subsets of $S$, hence \cite[Cor~8.11]{cox_in_hil} \color{black} shows that the group $\rho (W_S)$ preserves a round convex domain. Lemma~\ref{lem:round_case} now concludes the proof in case \eqref{item:round}. 
%To do so, one looks at the reflection of $W_S$ across hyperplanes containing $x$ and look at their action on the quotient vector space $V/x$, hence one gets a representation of the Coxeter group $W_{S \smallsetminus t} \to \mathrm{GL} (V/x)$ as a reflection group, where $t$ is the unique reflection across an hyperplane not containing $x$. 
	 
%	 \smallskip
%	 
%	 The Vinberg domain of $\Delta_x$ is $\Pb(V/x)$ if and only if $W_{S \smallsetminus t}$ is spherical. The Vinberg domain of $\Delta_x$ is an affine chart of $\Pb(V/x)$ if and only if $W_{S \smallsetminus t}$ is affine irreducible and $\Cart_{S \smallsetminus t}$ is of zero type. In that case, the group $\rho (W_{S \smallsetminus t})$ is conjugate to a virtually unipotent subgroup of $O(N-1,1)$, preserves an ellipsoid $\Ec$ of $\Pb (V)$, fixes a point $\zeta \in \partial \Ec$, preserves an horosphere based at $\zeta$ and acts cocompactly on that horosphere.  The Vinberg domain of $\Delta_x$ is a simplex of $\Pb(V/x)$ if and only if $W_{S \smallsetminus t}$ is affine irreducible and $\Cart_{S \smallsetminus t}$ is of negative type (in that case, one must have $W_{S \smallsetminus t} \simeq \widetilde{A}_{N-1}$).\note{I don't understand the role of this paragraph for the rest of the proof. \\ Ludo: Ok. This paragraph is formally useless but i thought it would help the reader to understand.}
	 
	 \smallskip
In case \eqref{item:not_round}, the proof of \cite[Lem.~3.22]{danciger2023convex} together with Lemma \ref{lem:restrict_to_irreducible_rep} show that $\rho (W_T)$ contains an element conjugate to $D(\lambda)$ for some $\lambda > 1$ and this element is the product of two reflections.
\end{proof}

\section{Lie groups containing a matrix $D(\lambda)$}\label{sec:lie-groups-containing-dlambda}

The goal of this section is to list the connected semisimple complex Lie subgroups of $\mathrm{GL}_n(\mathbb{C})$ acting irreducibly on $\mathbb{C}^n$ and containing a matrix of the form $D(\lambda)$ for some $\lambda \in \mathbb{C}^\times$ that is not a root of unity.

For $G = \SL_n(\Cb)$, $\textrm{SO}_n(\Cb)$, or, if $n$ is even, $\textrm{Sp}_{n}(\Cb)$, we call the representation $G \rightarrow \mathrm{GL}_n(\Cb)$, $A\mapsto A$ the \emph{defining representation}. By the \emph{dual defining representation} of $\SL_n(\Cb)$, we mean the representation $\SL_n(\Cb) \rightarrow \mathrm{GL}_n(\Cb)$,  $A\mapsto (A^\top)^{-1}$. We use the same names to refer to the associated representations on the level of Lie algebras. This section is dedicated to the proof of the following theorem.

\begin{theorem}\label{thm:rep_containing_d_lambda}
	Let $G$ be a connected semisimple complex Lie group and $\rho:G \to \GL_n (\Cb)$ be an irreducible faithful representation. Assume that $\rho(G)$ contains a diagonal matrix of the form~$D(\lambda)$, where $\lambda$ is not a root of unity.
Then $G$ is either $\SL_n(\Cb)$, $\SO_n(\Cb)$, or, if $n$ is even, $\mathrm{Sp}_{n}(\Cb)$. Furthermore, there is an automorphism $\phi\in\mathrm{Aut}(G)$ such that $\rho\circ\phi$ is conjugate to the defining representation.
\end{theorem}

Let $G$ be a complex semisimple Lie group. We denote by $\textrm{rk}(G)$ the rank of $G$. We will need the following terminology on representations. Let $\rho$ be a nontrivial finite-dimensional complex representation of $G$. One says $\rho$ is \emph{minuscule} if the Weyl group of $G$ acts transitively on the set of weights of~$\rho$, and \emph{quasi-minuscule} if the Weyl group of $G$ acts transitively on the set of nonzero weights of~$\rho$. When $G$ is simple, Theorem \ref{thm:rep_containing_d_lambda} will be a consequence of the following two propositions. 

%theorem.

%\begin{theorem}
%\label{Theorem_classification_Lie_theory}
%Let $G$ be a connected complex simple Lie group and $\rho:G\to\GL_n(\Cb)$ a faithful irreducible representation. Suppose that $\rho(G)$ contains a matrix $D(\lambda)$, where $\lambda$ is not a root of unity. If $n$ is odd, then $G$ is either $\SL_n(\Cb)$ or $\SO_n(\Cb)$. If $n$ is even, then $G$ is either $\SL_n(\Cb)$, $\SO_n(\Cb)$, or $\mathrm{Sp}_{n}(\Cb)$. Furthermore, there is an automorphism $\phi\in\textrm{Aut}(G)$ such that $\rho\circ\phi$ is conjugate to the defining representation.
%\end{theorem}

%\subsection{Trial...}

\begin{proposition}	\label{Proposition_Quasiminusculerepresentations}
	Let $G$ be a connected complex simple Lie group and $\rho:G\to\GL_n(\Cb)$ an irreducible representation. Suppose that $\rho(G)$ contains a matrix $D(\lambda)$, where $\lambda$ is not a root of unity. Then  $\rho$ is a quasi-minuscule representation with at most $2\mathrm{rk}(G)+1$ weights. Furthermore, if $\rho$ is minuscule, then $n\leq 2\mathrm{rk}(G)$.
\end{proposition}

\begin{proposition}
	\label{Proposition_classi_quasi-minus_and_minus}
	Let $G$ be a connected complex simple Lie group and $\rho:G\to\GL_n(\Cb)$ an irreducible faithful representation. Suppose that $\rho$ is a quasi-minuscule representation with at most $2\mathrm{rk}(G)+1$ weights. If $\rho$ is minuscule, assume further that $n\leq 2\mathrm{rk}(G)$. Then $G$ is either $\SL_n(\Cb)$, $\SO_n(\Cb)$, or, if $n$ is even, $\mathrm{Sp}_{n}(\Cb)$. Furthermore, there is an automorphism $\phi\in\mathrm{Aut}(G)$ such that $\rho\circ\phi$ is conjugate to the defining representation.
\end{proposition}

Note that, conversely, the defining representations of $\SL_n(\Cb)$, $\SO_n(\Cb)$, and $\mathrm{Sp}_{2m}(\Cb)$ always contain a matrix of the form $D(\lambda)$, where $\lambda$ is not a root of unity.
We proceed to the proof of Proposition \ref{Proposition_Quasiminusculerepresentations}, for which we need the following lemma.

\begin{lemma}
\label{lemmadimensionfinitegroupaction}
Let $W$ be a finite group and $V$ an irreducible linear representation of $W$. Suppose there exist $v\in V\setminus\{0\}$ and $v'\in Wv$ such that $Wv\setminus\{v,v'\}$ does not span $V$. Then $$|Wv|\leq 2 \dim (V).$$
\end{lemma}

\begin{proof}
Let $n$ be the dimension of $V$. Assume that $Wv\setminus\{v,v'\}$ is contained in the kernel of a nonzero linear form $l\in V^*$. Consider the dual representation of $W$ on $V^*$. Since the latter is irreducible, $Wl$ spans $V^*$. Hence, there exist $g_1,\dots,g_n\in W$ such that $(g_1l,\dots,g_nl)$ is a basis of $V^*$. Let $1\leq i\leq n$. Then every vector in $Wv$, except possibly $g_iv$ and $g_iv'$, lies in $\ker(g_il)$. Hence $|\ker(g_il) \cap Wv|\geq |Wv|-2$. It follows that $0=|\cap_i\ker(g_il) \cap Wv|\geq |Wv|-2n$.
\end{proof}

%\begin{proposition}
%\label{Proposition_Quasiminusculerepresentations}
%Let $G$ be a connected complex simple Lie group and $\rho:G \to \GL_n (\Cb)$ an irreducible representation. Suppose that $\rho(G)$ contains a matrix $D(\lambda)$, where $\lambda$ is not a root of unity. Then $\rho$ is a quasi-minuscule representation with at most $2\mathrm{rk}(G)+1$ weights. If $\rho$ is moreover minuscule, then $n\leq 2\mathrm{rk}(G)$.
%\end{proposition}

\begin{proof}[Proof of Proposition~\ref{Proposition_Quasiminusculerepresentations}]
First, note that $\rho(G)$ is the identity component of its own Zariski-closure (as follows for instance from \cite[Lemma 7.9]{Borel_Linearalgebraicgroups}).
Since $\rho(G)$ contains $D({\lambda})$, we have that $\rho(G)$ contains the infinite cyclic subgroup generated by $D({\lambda})$. The latter subgroup has Zariski-closure $\{D(x) |\ x\in\Cb^{\times}\}$, which is connected and hence lies in $\rho(G)$. Let $\rho:\mathfrak{g}\to\mathfrak{gl}_n (\Cb)$ be the representation induced by $\rho$ on the level of Lie algebras. Then $\rho(\mathfrak{g})$ contains an element of the form $\textrm{diag}(1,0,\dots,0,-1)$. Denote by $X$ a preimage of this element under $\rho$.

By \cite[Theorem 3 in Chapter 1, \S 3]{Bourbaki_LiegroupsChapter13}, we have that $X$ is a semisimple element of $\mathfrak{g}$. Hence $X$ is contained in a Cartan subalgebra $\mathfrak{h}$ of $\mathfrak{g}$ (see, for instance, \cite[Proposition 10 and Corollary 1 in  Chapter 10, \S3]{Bourbaki_LiegroupsandLiealgebras79}). Denote by $\Phi$ the set of roots of $\mathfrak{g}$ with respect to $\mathfrak{h}$ and by~$W$ the corresponding Weyl group. Pick a Weyl chamber in $\mathfrak{h}$ whose closure contains~$X$. We choose the set of roots which are positive on this Weyl chamber as our set of positive roots.

The action of $\mathfrak{h}$ on $\Cb^n$ induces the decomposition $$\Cb^n=\bigoplus_{\omega\in\Pi} V_{\omega},$$ where $\Pi\subset\mathfrak{h}^*$ is the set of weights of $\rho$ and $V_{\omega}$ is the weight space associated to $\omega$. We have $\omega(X) = 0$ for all weights $\omega \in \Pi$ except for precisely two weights $\omega^+,\omega^-\in \Pi$ satisfying $\omega^+(X)=1$ and $\omega^-(X)=-1$. Since $\alpha(X)\geq 0$ for all positive roots $\alpha$, we have that $\omega^+$ must be the highest weight of $\rho$ and $\omega^-$ the lowest weight.
Recall that $\Pi$ is preserved under $W$ (see \cite[Theorem in \S21.2]{Humphreys_IntroductiontoLiealgebras}). An element of $ W$ mapping $\Phi$ to $-\Phi$ reverses the order on $\Pi$ and hence sends $\omega^+$ to $\omega^-$. We conclude that $\omega^-$ is in the $W$-orbit of $\omega^+$.

Now suppose $\theta\in\Pi$ is not in the $W$-orbit of $\omega^+$. By the previous paragraph, we have that~$W\theta$ contains neither $\omega^+$ nor $\omega^-$. Thus, any element of $W\theta$ vanishes on $X$, so in particular~$W\theta$ does not span $\mathfrak{h}^*$. Since $W$ acts irreducibly on $\mathfrak{h}^*$ (see \cite[Lemma B in \S 10.4]{Humphreys_IntroductiontoLiealgebras}\footnote{This lemma actually states that $W$ acts irreducibly on the real span of the roots, but the proof indeed shows that $W$ acts irreducibly on $\mathfrak{h}^*$.}), we conclude that $W\theta$ must be $\{0\}$. Hence $\rho$ is quasi-minuscule.

Since $W\omega^+\setminus\{\omega^+,\omega^-\}$ does not span $\mathfrak{h}^*$, Lemma \ref{lemmadimensionfinitegroupaction} shows that $|W\omega^+|\leq 2\dim(\mathfrak{h}^*)=2\textrm{rk}(G)$. It follows that $\rho$ has at most $2\textrm{rk}(G)+1$ weights. Now assume that $\rho$ is minuscule. Since each weight space has dimension 1 (see \cite[Proposition 7 in Chapter 8, \S7]{Bourbaki_LiegroupsandLiealgebras79}), $|\Pi|=n$. Hence $n=|\Pi| = |W\omega^+|\leq 2\dim(\mathfrak{h}^*)=2\textrm{rk}(G)$.
\end{proof}

%
%By Proposition \ref{Proposition_Quasiminusculerepresentations}, we know that $\rho$ must be quasi-minuscule. Assume first that $\rho$ is minuscule. Again, by Proposition \ref{Proposition_Quasiminusculerepresentations}, $n\leq 2\textrm{rk}(G)$. Minuscule representations are classified in \cite[Chapter 8, \S7.3]{Bourbaki_LiegroupsandLiealgebras79} and one can find their dimensions in \cite[Table 2 in Chapter 8]{Bourbaki_LiegroupsandLiealgebras79}.

\begin{proof}[Proof of Proposition~\ref{Proposition_classi_quasi-minus_and_minus}]
Assume first that $\rho$ is minuscule and $n\leq 2\textrm{rk}(G)$. Minuscule representations are classified in \cite[Chapter 8, \S7.3]{Bourbaki_LiegroupsandLiealgebras79} and one can find their dimensions in \cite[Table 2 in Chapter 8]{Bourbaki_LiegroupsandLiealgebras79}.

The only minuscule representations of $\mathfrak{sl}_{l+1}(\Cb)$ of dimension $\leq 2l$ are the defining representation, the dual defining representation, or, if $l=3$, the fundamental representation associated to the middle root. The latter is the defining representation of $\mathfrak{so}_6(\Cb)$ \cite[Part 3, \S19.1]{Fulton_RepresentationTheory}. For $\mathfrak{so}_{2l+1}(\Cb)$, $l\geq 2$, the unique minuscule representation has low enough dimension if and only if $l=2$. In the latter case, we obtain the defining representation of $\mathfrak{sp}_4(\Cb)$ \cite[Part 3, \S244]{Fulton_RepresentationTheory}. The unique minuscule representation of $\mathfrak{sp}_{2l}(\Cb)$ is the defining representation. For $\mathfrak{so}_{2l}(\Cb)$, $l\geq 3$, the minuscule representations are precisely the defining representation and the spinor representations. The spinor representations have dimension $2^{l-1}$; the latter is smaller than $2l$ if and only if $l=3$ or $l=4$. If $l=3$, the spinor representations are the defining representation of $\mathfrak{sl}_4(\Cb)$ and its dual. If $l=4$, the spinor representations are the composition of the defining representation of $\mathfrak{so}_8 (\Cb)$ with a triality automorphism \cite[Part 3, \S20.3]{Fulton_RepresentationTheory}. The minuscule representations of $\mathfrak{e}_6$ and $\mathfrak{e}_7$ are of dimension strictly larger than twice the rank. Finally, the exceptional Lie algebras $\mathfrak{g}_2$, $\mathfrak{f}_4$, and $\mathfrak{e}_8$ have no minuscule representations.

Now assume that $\rho$ is quasi-minuscule but not minuscule, so that $0\in\Pi$. Then there is a root $\alpha\in\Pi$; see \cite[Proposition 5 ii) in Chapter 8, \S7]{Bourbaki_LiegroupsandLiealgebras79}. Since the action of $W$ on the set of roots of the same length as $\alpha$ is transitive \cite[Lemma B in \S10.4]{Humphreys_IntroductiontoLiealgebras}, we have that~$\Pi$ contains all such roots. If all roots have the same length, then $\Pi=\Phi\cup\{0\}$. Otherwise $\alpha$ must be a short root. Indeed, if $\alpha$ were a long root, then $\Pi$ would also contain a short root. However, short and long roots are never in the same orbit for the Weyl group, since the latter acts isometrically. Hence $\Pi$ is the set $\Phi_{\mathrm{short}}$ of short roots together with $0$.

The description of the root systems is given in \cite[Plate 1-9]{Bourbaki_LiegroupsandLiealgebras46}. If all roots of $\mathfrak{g}$ have the same length then one can check that $|\Phi|+1 > 2\textrm{rk}(G)+1$ unless $\mathfrak{g}\simeq \mathfrak{sl}_2(\mathbb{C})$. In the latter case, $\rho$ is the adjoint representation which is the defining representation of $\mathfrak{so}_3(\mathbb{C})$.
If $\mathfrak{g}$ has short and long roots then one can check that $|\Phi_{\mathrm{short}}| \leqslant 2 \textrm{rk} (G)$ if and only if $\mathfrak{g} = \mathfrak{so}_{2l+1}(\Cb)$ for some $l \geqslant 1$. The only quasi-minuscule representation of $\mathfrak{so}_{2l+1}(\Cb)$ with $\Pi = \Phi_{\mathrm{short}} \cup \{0 \}$ is the defining representation \cite[Part 3, \S19.4]{Fulton_RepresentationTheory}.
\end{proof}

\begin{proof}[Proof of Theorem \ref{thm:rep_containing_d_lambda}]
First, $\rho(G)\subset\SL_n(\Cb)$ since the one-dimensional representation $\det\circ\rho:G\to\mathbb{C}^{\times}$ must be trivial. Define $$H:=\left\{\begin{array}{ll}
\SL_n(\Cb) & \mbox{if $\rho(G)$ does not preserve any nondegenerate bilinear form}\\
\SO_n(B) & \mbox{if $\rho(G)$ preserves a nondegenerate symmetric bilinear form $B$}\\
\Sp_n(B) & \mbox{if $\rho(G)$ preserves a nondegenerate antisymmetric bilinear form $B$.}
\end{array}\right.
$$
We want to show that $\rho(G)=H$. Suppose otherwise. Then $\rho(G)$ is contained in a maximal proper connected Lie subgroup $H_0$ of $H$. Since $H_0$ contains $\rho(G)$, we have that $H_0$ acts irreducibly on $\mathbb{C}^n$ and contains a matrix of of the form $D(\lambda)$, where $\lambda$ is not a root of unity.

Suppose that $H_0$ is simple. Proposition~\ref{Proposition_Quasiminusculerepresentations} and Proposition~\ref{Proposition_classi_quasi-minus_and_minus} show that $H_0$ is a classical group and $\rho$ is conjugate to the defining representation of $H_0$ up to an automorphism. In particular $H_0=H$, which contradicts the definition of $H_0$.

Suppose that $H_0$ is not simple. Then one concludes from \cite[Theorem 1.3 and Theorem 1.4]{Dynkin_Maximalsubgroups} that $H_0$ is, up to conjugation, contained in $\SL_{n_1}(\Cb)\otimes\SL_{n_2}(\Cb)$ for some $n_i\geq2$ satisfying $n_1n_2=n$. Hence there exist $A\in\SL_{n_1}(\Cb)$ and $B\in\SL_{n_2}(\Cb)$ such that $A\otimes B$ has eigenvalues $\lambda,1,\dots,1,\lambda^{-1}$. Denote by $\xi_1,\dots,\xi_{n_1}$ the eigenvalues of $A$ and by $\mu_1,\dots,\mu_{n_2}$ those of $B$. Then the eigenvalues of $A\otimes B$ are $(\xi_i \mu_j)_{i,j}$. We can assume that $\xi_1\mu_1=\lambda$. If $n_1\geq 3$, there is an $i$ such that $\xi_i\mu_j=1$ for all $j$. This shows that $\mu_j=\mu_{j'}$ for all $j$ and $j'$. We conclude that $\xi_1\mu_j=\lambda$ for all $j$, a contradiction. The same argument applies if $n_2\geq 3$. Hence, $n_1=n_2=2$. In this case, we have by \cite[Theorem 1.3 and Theorem 1.4]{Dynkin_Maximalsubgroups} that $H_0$ is, up to conjugation, equal to $\SL_2(\Cb)\otimes\SL_2(\Cb)$. The latter group preserves the symmetric bilinear form $\omega\otimes\omega$, where~$\omega$ is a nondegenerate antisymmetric bilinear form on $\Cb^2$. In particular, $H_0$ is conjugate to $\SO_4(\Cb)$. This concludes the proof.% \color{blue} The group \note{Ludo: the proof need a fix. Add this sentence to finish the proof.} $H_0$ is a connected subgroup of $\SO_4(\Cb)$ acting irreducibly on $\mathbb{C}^4$. Since $H_0$ is semisimple, for dimension reason either $H_0$ is isomorphic to $\SL_2 (\Cb)$ (but this is excluded by the hypothesis $H_0$ is not simple) or $H_0 = \SO_4 (\Cb)$.  \Bk
\end{proof}

\section{Proof of Theorem \ref{thm:coxeter}}\label{sec:proof-of-main}

To prove Theorem \ref{thm:coxeter}, we will need some preliminaries on Zariski-closures of absolutely irreducible representations.

\subsection{Semisimplicity of the Zariski-closure}

\begin{lemma}\label{lemma_semisimple}
	Let $\Gamma<\SL^{\pm}(V)$ be a subgroup such that the action of any finite-index subgroup of $\Gamma$ on $V$ is absolutely irreducible. Then the Zariski-closure of $\Gamma$ in $\SL^{\pm}(V_\Cb)$ is semisimple.
\end{lemma}

Here, we consider $\SL^{\pm}(V_\Cb)$ as a complex algebraic group.

\begin{proof}
	Denote by $G$ the identity component, for the Euclidean topology, of the Zariski-closure of $\Gamma$ in $\SL^{\pm}(V_\Cb)$, and by $\mathfrak{g}$ the Lie algebra of $G$. Since $G$ contains a finite-index subgroup of~$\Gamma$, the action of $G$ (and hence of $\mathfrak{g}$) on $V_{\Cb}$ is irreducible. It follows that $G$ is reductive with radical the center of $G$ \cite[Proposition 5 in Chapter 1, \S 6]{Bourbaki_LiegroupsChapter13}. Since $G \subset \SL^{\pm}(V_\Cb)$, the center of $G$ is finite, and hence $\mathfrak{g}$ has no center. We conclude that $\mathfrak{g}$ is semisimple \cite[Theorem 1 in Chapter 1, \S 6]{Bourbaki_LiegroupsChapter13}.
\end{proof}

\subsection{Invariant bilinear forms on the complexification of $V$}

We will need a strengthening of Theorem~\ref{thm:tits-vinberg2}.\eqref{item:bilinear_form} in the complex case. This is the objective of this section.

\begin{lemma}\label{lemma_invariantbilinearform}
	Let $\Gamma<\GL (V)$ be a subgroup generated by reflections such that the action of $\Gamma$ on $V$ is absolutely irreducible. Every nonzero bilinear form on $V_{\mathbb{C}}$ which is preserved by $\Gamma$ is nondegenerate and symmetric.
\end{lemma}

\begin{proof}
	Let $B$ be such a form. Let $x\in\ker(B)$. For all $y\in V_{\mathbb{C}}$ and $g\in\Gamma$, we have $B(g x,y)=B(x,g^{-1}y)=0.$ Hence $\ker(B)$ is $\Gamma$-invariant. Since $B$ is nonzero, it is nondegenerate.
	
	One can write $B$ uniquely as a sum $B=B_s+B_a$ where $B_s$ is symmetric and $B_a$ is antisymmetric. Since $B$ is $\Gamma$-invariant, $B_s$ and $B_a$ are $\Gamma$-invariant. Let $\sigma\in\Gamma$ be a reflection. Write $\sigma = \Id -\alpha\otimes v$ for $\alpha\in V_{\mathbb{C}}^*$ and $v\in V_\mathbb{C}$ such that $\alpha(v)=2$. For all $x,y\in V_{\mathbb{C}}$,
	\begin{align*}
		&B_a(x,y)=B_a(\sigma(x),\sigma(y))\\
		\Leftrightarrow\ &B_a(x,y)=B_a(x,y)-\alpha(x)B_a(v,y)-\alpha(y)B_a(x,v)+\alpha(x)\alpha(y)B_a(v,v)\\
		\Leftrightarrow\ &\alpha(x)B_a(v,y)=\alpha(y)B_a(v,x).
	\end{align*}
	Taking $y=v$, we obtain that $B_a (v,x) = 0$ for every $x \in V_\Cb$. Hence, $v$ lies in $\ker (B_a)$ which, by the previous paragraph, shows that $B_a=0$.
\end{proof}

\begin{lemma}
	\label{lemmainvriantbilinearformsandfiniteindex}
	Let $\Gamma<\GL (V)$ be a subgroup generated by reflections such that the action of any finite-index subgroup of $\Gamma$ on $V$ is absolutely irreducible. Suppose that there exists a finite-index subgroup $\Gamma_0$ of $\rho (W)$ that preserves a nonzero $\Cb$-bilinear form on $V_\Cb$. Then $\rho(W)$ preserves a nonzero symmetric $\mathbb{R}$-bilinear form on $V$.
\end{lemma}

\begin{proof}
	Denote by $B$ a nonzero $\Cb$-bilinear form on $V_{\Cb}$ preserved by $\Gamma_0$. Up to diminishing~$\Gamma_0$, we may assume that $\Gamma_0$ is normal in $\Gamma:= \rho (W)$. By assumption, we have that $\Gamma_0$ acts irreducibly on $V_{\Cb}$.
	
	Since the kernel of $B$ is invariant under $\Gamma_0$, we have that $B$ is nondegenerate. Since $\Gamma$ normalizes $\Gamma_0$, for all $g\in \Gamma_0$, $\gamma\in\Gamma$ and $x,y\in V_{\Cb}$, we have
	$$B(\gamma g x,\gamma g y)=B(\gamma x,\gamma y).$$ Hence for all $\gamma\in\Gamma$, we have that $\Gamma_0$ preserves the bilinear form $(x,y)\mapsto B(\gamma x,\gamma y)$. Since the action of $\Gamma_0$ on $V_{\Cb}$ is irreducible, all $\Gamma_0$-invariant bilinear forms on $V_{\Cb}$ are scalar multiples of one another. We deduce that for each $\gamma\in\Gamma$, there exists a nonzero $c_{\gamma}\in\Cb$ such that $$B(\gamma\cdot,\gamma\cdot)=c_{\gamma}B(\cdot,\cdot).$$ The map $c:\Gamma\to\Cb^{*},\ \gamma\mapsto c_{\gamma}$ is a group homomorphism.
	
	We claim that $c$ is trivial. Otherwise, there exists a reflection $\sigma\in\Gamma$ such that $c_{\sigma}\neq 1$. Write $\sigma=\mathrm{Id}-\alpha\otimes v$ for $\alpha\in V_{\Cb}$ and $v\in V_{\Cb}^*$ satisfying $\alpha(v)=2$. For all $x,y\in\ker(\alpha)$, $$c_{\sigma}B(x,y)=B(\sigma(x),\sigma(y))=B(x,y),$$ which implies that $B(x,y)=0$. In particular, if $x$ and $y$ are nonzero, we have $x^{\perp}=\ker(\alpha)=y^{\perp}$ which shows that $x$ and $y$ are colinear. Hence $\dim(\ker(\alpha))=1$, so that $V$ has dimension~$2$.
	However, in the latter case, the Cartan matrix of $\rho$ is always symmetrizable, so that $\Gamma$ preserves a nonzero bilinear form on $V$. By irreducibility, this form must be a multiple of $B$. In any case, we obtain that $c$ is trivial, so that $\Gamma$ preserves $B$.

%	By Lemma \ref{lemma_invariantbilinearform}, $B$ must be symmetric, and by Lemma \ref{lemma_symmetrizable} the Cartan matrix of $\rho$ is symmetrizable. Theorem \ref{thm:tits-vinberg2} Item \ref{item:bilinear_form} concludes that  $\rho$ preserves a nonzero symmetric $\mathbb{R}$-bilinear form on $V$.

	By Lemma \ref{lemma_invariantbilinearform}, $B$ must be symmetric. The symmetric $\mathbb{R}$-bilinear form on $V$ given by $(x,y)\mapsto B(x,y)+\overline{B(x,y)}$ is preserved by $\Gamma$. If the latter form is nonzero, the conclusion of the lemma holds. Otherwise, the restriction of $iB$ to $V$ is a nonzero symmetric $\mathbb{R}$-bilinear form preserved by $\Gamma$. 
\end{proof}

\subsection{The proof}

We can now prove Theorem \ref{thm:coxeter}.

\begin{proof}[Proof of Theorem \ref{thm:coxeter}]
Let $W$ be a Coxeter group that is not virtually abelian (in which case~$W$ is large). Let $\rho:W\to\GL(V)$ be an irreducible representation of $W$ as a reflection group; note that irreducibility of $\rho$ implies that $W$ is irreducible as an abstract Coxeter group. Let $G$ be the identity component (for the Euclidean topology) of the Zariski-closure of $\Gamma:=\rho(W)$ in $\SL^{\pm}(V_{\mathbb{C}})$ (where $\SL^{\pm}(V_{\mathbb{C}}$) is considered as a complex algebraic group).

Since the restriction of $\rho$ to any finite-index subgroup of $\Gamma$ is absolutely irreducible (see Theorem \ref{thm:tits-vinberg2}, Item (\ref{item:strongly_irreducible})), we have that $G$ acts irreducibly on $V_{\mathbb{C}}$. Lemma \ref{lemma_semisimple} shows that $G$ is semisimple. 
By Proposition \ref{prop_dlambda}, there exists $\gamma \in \Gamma$ conjugate within $\GL(V)$ to $D({\lambda})$ for some $\lambda > 1$. A positive power of $\gamma$ lies in $G$. We can thus apply Theorem \ref{thm:rep_containing_d_lambda}, which shows that $G$ is conjugate to either $\SL(V_{\Cb})$, $\SO_B(V_{\Cb})$, or $\mathrm{Sp}_B(V_{\Cb})$ for some nondegenerate bilinear form $B$ on~$V_{\Cb}$.

 Denote by $H$ the Zariski-closure of $\Gamma$ in $\SL^{\pm}(V)$ (considered as a real algebraic group). Note that $G$ is the identity component of the complex points of $H$.
Suppose that $\Gamma$ preserves a nonzero symmetric $\mathbb{R}$-bilinear form $f$ on $V$. Then $H$ is contained in $O_f(V)$. By the previous paragraph, $G=\SO_B(V_{\Cb})$, where $B$ is the complexification of $f$. This shows that $\SO_f(V)\subset H$. Since $H$ contains a reflection, we have $H=O_f(V)$ in this case. Now suppose that $\Gamma$ does not preserve any nonzero symmetric $\mathbb{R}$-bilinear form on $V$. Lemma \ref{lemmainvriantbilinearformsandfiniteindex} thus shows that $G \cap \Gamma$ does not preserve any nonzero $\Cb$-bilinear form on $V_{\Cb}$. By the previous paragraph, we have $G=\SL(V_{\Cb})$, so that $\SL(V)\subset H$. Since $H$ contains a reflection, we have $H=\SL^{\pm}(V)$ in this case. This concludes the proof.
\end{proof}

\section{Thin reflection groups}\label{sec:thin-reflection-groups}

In this section, we apply Theorem \ref{thm:coxeter} to construct thin (virtual) embeddings of certain abstract Coxeter groups. We will use the following notion.

\subsection{Cyclic products}

Given a matrix $\mathcal{A}$, a \emph{cyclic product of length $k\geq 2$ in} $\mathcal{A}$ is a product of the form $$\mathcal{A}_{i_1i_2}\mathcal{A}_{i_2i_3}\dots\mathcal{A}_{i_ki_1},$$
where $i_1, i_2, \dotsc, i_k$ are distinct indices.
The relevance of cyclic products is demonstrated by the following proposition.

\begin{proposition}\emph{\cite[Prop. 20]{the_bible}}
\label{prop:cyclic_products_equivalence}
A Cartan matrix $\mathcal{A}$ is symmetrizable if and only if
$$\mathcal{A}_{i_1i_2}\mathcal{A}_{i_2i_3}\dots\mathcal{A}_{i_ki_1} = \mathcal{A}_{i_2i_1} \mathcal{A}_{i_3i_2}\dots\mathcal{A}_{i_1 i_k}$$
for all distinct indices $i_1, i_2, \dotsc, i_k$.
\end{proposition}

%\begin{proposition}{\cite[Propositions 16 and 20]{the_bible}}
%\label{prop:cyclic_products_equivalence}
%\begin{enumerate}
%\item Two Cartan matrices are equivalent if and only if all their cyclic products agree.
%\item A Cartan matrix $\mathcal{A}$ is symmetrizable if and only if
%$$\mathcal{A}_{i_1i_2}\mathcal{A}_{i_2i_3}\dots\mathcal{A}_{i_ki_1} = \mathcal{A}_{i_2i_1} \mathcal{A}_{i_3i_2}\dots\mathcal{A}_{i_1 i_k}$$
%for all distinct indices $i_1, i_2, \dotsc, i_k$.
%\end{enumerate}
%\end{proposition}

%In particular, given a Cartan matrix $\Cart$, if there exists distinct indices $i_1,i_2,\dots,i_k$ such that $$\mathcal{A}_{i_1i_2}\mathcal{A}_{i_2i_3}\dots\mathcal{A}_{i_ki_1}\neq\mathcal{A}_{i_1i_k}\mathcal{A}_{i_ki_{k-1}}\dots\mathcal{A}_{i_2i_1}$$ then $\Cart$ is not symmetrizable.

The following lemma, due to Vinberg \cite{Vinberg_Rings} (see also \cite{ChoiChoi}), gives necessary and sufficient conditions for an irreducible Vinberg representation to have image in a conjugate of~$\mathrm{GL}_n(\mathbb{Z})$. We include a proof for the convenience of the reader. 

\begin{lemma}\label{lem:cyclic-products-integers}
Let $W$ be a large irreducible Coxeter group and $\rho:W\to\mathrm{GL}_n(\mathbb{R})$ be an irreducible Vinberg representation with Cartan matrix $\mathcal{A}$. There is a conjugate of $\rho$ with image in $\mathrm{GL}_n(\mathbb{Z})$ if and only if all cyclic products in $\mathcal{A}$ are integers.
\end{lemma}

\begin{proof}
Suppose that $\rho$ has a conjugate with image in $\mathrm{GL}_n(\mathbb{Z})$. We can assume that this conjugate is $\rho$ itself. Let $i_1,\dots,i_k$ be indices between $1$ and $\mathrm{rank}(W)$. The trace of the product \begin{equation}
\label{equationcyclicproduct}
(I_n-\rho(s_{i_1}))(I_n-\rho(s_{i_2}))\dots(I_n-\rho(s_{i_k}))\in\mathrm{M}_n(\mathbb{Z})\end{equation}
is the cyclic product $\alpha_{i_1}(v_{i_2})\alpha_{i_2}(v_{i_3})\dots\alpha_{i_k}(v_{i_1})$, and the latter is therefore an integer.

Conversely, suppose that all cyclic products in $\mathcal{A}$ are integers. For a subring $R\subset \mathbb{C}$, denote by $R\rho(W)$ the $R$-span of $\rho(W)$ in $\mathrm{M}_n(\mathbb{C})$. Theorem \ref{thm:tits-vinberg2}, Item (\ref{item:strongly_irreducible}), shows that $\rho$ is absolutely irreducible, hence $\mathbb{C}\rho(W)=\mathrm{M}_n(\mathbb{C})$ \cite[Corollary 3.4, Chapter XVII]{Lang_Algebra}. Since $\mathbb{R}\rho(W)\subset\mathrm{M}_n(\mathbb{R})$ has complexification $\mathrm{M}_n(\mathbb{C})$, we have $\mathbb{R}\rho(W)=\mathrm{M}_n(\mathbb{R})$.

As an algebra, $\mathbb{Z}\rho(W)$ is generated by the elements $I_n-\rho(s_i)$ for $1\leq i\leq \mathrm{rank}(W)$. Hence, every element of $\mathbb{Z}\rho(W)$ is a $\mathbb{Z}$-linear combination of products of the form \eqref{equationcyclicproduct}. Since, all cyclic products of $\mathcal{A}$ are integers, the trace of any element in $\mathbb{Z}\rho(W)$ is an integer.

Thanks to \cite[Lemma 1.2(b)]{Bass_Groupsofintegralrepresentationtype}, there exists a basis $(h_i)_{1\leq i\leq n^2}$ of $\mathrm{M}_n(\Cb)$ such that $\mathbb{Q}\rho(W)\subset \bigoplus_{i=1}^{n^2}\mathbb{Q}h_i.$
This shows that the $\mathbb{Q}$-algebra $\mathbb{Q}\rho(W)$ has complexification $\mathrm{M}_n(\mathbb{C})$. By Wedderburn's Theorem \cite[Theorem 2.1.3]{Gille_Centralsimplealgebras}, there is an isomorphism between $\mathbb{Q}\rho(W)$ and $\mathrm{M}_m(D)$, where $D$ is a division algebra over $\mathbb{Q}$ of dimension $d$ and $n=dm$. This isomorphism induces an automorphism of $\mathrm{M}_n(\mathbb{R})$, which is necessarily inner by the Skolem--Noether Theorem \cite[Theorem 2.7.2]{Gille_Centralsimplealgebras}. Hence there exists an element $g\in\mathrm{GL}_n(\mathbb{R})$ such that $g\mathbb{Q}\rho(W)g^{-1}=\mathrm{M}_m(D)$.
The rank of any element of $\mathrm{M}_m(D)$ is a multiple of $d$. Since $\mathbb{Q}\rho(W)$ contains $I_n-\rho(s_1)$ which has rank $1$, we conclude that $d=1$ and that $g\mathbb{Q}\rho(W)g^{-1}=\mathrm{M}_n(\mathbb{Q}).$

Up to conjugation, we can now assume that $\rho$ has image in $\GL_n(\mathbb{Q})$. 
Pick elements $A_1,\dots,A_{n^2}$ of $\mathbb{Z}\rho(W)$ which form a basis of $\mathbb{Q}\rho(W)$. The map $\mathbb{Z}\rho(W)\to\mathbb{Z}^{n^2}$ given by $\ X\mapsto (Tr(A_iX))_{1\leq i\leq n^2}
$ is well-defined and is an injection of $\mathbb{Z}$-modules. This shows that $\mathbb{Z}\rho(W)$ is finitely generated as a $\mathbb{Z}$-module.
Let $e_1,\dots,e_n$ be the canonical basis of $\mathbb{Q}^n$. Define $$L=\sum_{i=1}^n\mathbb{Z}\rho(W)e_i\subset \mathbb{Q}^n.$$ This is a finitely generated $\mathbb{Z}$-module which spans $\mathbb{Q}^n$. Hence $L$ is a free $\mathbb{Z}$-module \cite[Theorem 7.3, Chapter III]{Lang_Algebra}. Since $L$ is $W$-invariant, there exists a basis of $L$ in which $\rho$ has only integer entries.
\end{proof}

The applicability of Theorem \ref{thm:coxeter} to the construction of thin subgroups of $\mathrm{SL}_n(\mathbb{Z})$ is summarized in the following corollary.

\begin{corollary}\label{cor:integer_representation}
Let $W_S$ be a large irreducible Coxeter group and suppose that $\mathcal{A}$ is a nonsymmetrizable Cartan matrix compatible with $W_S$ all of whose cyclic products are integers. Then there is a Zariski-dense representation $\rho:W_S\to\mathrm{SL}^\pm_n(\mathbb{Z})$ of $W_S$ as a reflection group, where $n = \mathrm{rank}(\mathcal{A})$.
\end{corollary}

\begin{proof}
Consider the representation $\rho_\Cart : W_S \rightarrow \mathrm{GL}(V)$ of $W_S$ as a reflection group associated to~$\Cart$, where $V = \mathbb{R}^{|S|}$ (see Example \ref{eg:representation_from_cartan}). Following \S\ref{subsection:Invariantsubspace}, an identification of $V_v / (V_v \cap V_\alpha)$ with~$\mathbb{R}^n$ yields an irreducible representation $\rho:= (\rho_\Cart)_v^\alpha : W_S \rightarrow \mathrm{GL}_n(\mathbb{R})$ of $W_S$ as a reflection group with Cartan matrix $\Cart$. By Theorem \ref{thm:tits-vinberg2}, Item (\ref{item:bilinear_form}), the representation $\rho$ does not preserve any nonzero symmetric bilinear form on $\mathbb{R}^n$ as $\Cart$ is not symmetrizable. Thus, the Zariski-closure of $\rho(W_S)$ is $\mathrm{SL}_n^\pm(\mathbb{R})$ by Theorem \ref{thm:coxeter}. Moreover, by Lemma \ref{lem:cyclic-products-integers}, the representation $\rho$ is conjugate within $\mathrm{GL}_n(\mathbb{R})$ to a representation with image in $\mathrm{GL}_n(\mathbb{Z})$ since all the cyclic products in $\Cart$ are integers.
\end{proof}

\subsection{Right-angled thin reflection groups}
The goal of this section is to prove Theorem~\ref{thm:virtually-Zariski-dense}. We first consider the setting where the Coxeter diagram of $W_S$ is not a tree, and, along the way, obtain Theorem \ref{thm:surfacesubgroups}.

%\begin{proposition}
%{\color{red} Conjecture} Let $W$ be a Coxeter group with diagram $\mathcal{G}_W$. Let $e$ be an edge of $\mathcal{G}_W$ such that all adjacent edges to $e$ have labels $3,4,6$ or $\infty$ and such that the diagram obtained by removing $e$ (including endpoints) and all adjacent edges admits an integral Cartan matrix of full rank. Then $W$ admits an integral Cartan matrix of full rank.
%
%Further more, if $\mathcal{G}_W$ has a cycle not of the form $\Tilde{A}_k$, then $W$ admits an integral nonsymmetrizable Cartan matrix of full rank.
%\end{proposition}
%
%It implies Theorem \ref{thm:thin}.

\begin{proposition}\label{prop:thin}
	An irreducible right-angled Coxeter group $W_S$ whose Coxeter diagram is not a tree admits a Zariski-dense representation $W_S \rightarrow \mathrm{SL}_n^\pm(\mathbb{Z})$ as a reflection group, where $n = |S|$. 
\end{proposition}

\begin{proof}
Let $s_1, \ldots, s_n$ be an enumeration of $S$, where we identify the latter set with the vertex set of the Coxeter diagram $\mathscr{G}_{W_S}$ of $W_S$. Choose a spanning tree $\mathscr{T}$ in $\mathscr{G}_{W_S}$, and for $t \in \mathbb{R}$, let~$\mathscr{A}_t$ be the $n \times n$ real matrix whose $(i,j)^\text{th}$ entry $(\mathscr{A}_t)_{ij}$ is given by
\[
(\mathscr{A}_t)_{ij} = \begin{cases} 2 & i=j \\ 0 & i \neq j, \text{ and $s_i$ and $s_j$ are not adjacent in  $\mathscr{G}_{W_S}$} \\ -2t & i < j, \text{ and $s_i$ and $s_j$ are  adjacent  in $\mathscr{G}_{W_S}$} \\ -2t & i > j,  \text{ and $s_i$ and $s_j$ are  adjacent in $\mathscr{T}$} \\ -3t & \text{otherwise}.  \end{cases}
\]
For $t>0$, the Cartan matrix $\mathscr{A}_t$ is not symmetrizable since $\mathscr{T} \neq \mathscr{G}_{W_S}$. Since $\mathscr{A}_0 = 2 I_n$, we have that $\det(\mathscr{A}_0) \neq 0$, and so the polynomial $\det(\mathscr{A}_t)$ in $t$ is not the constant polynomial $0$. Thus, for a sufficiently large positive integer $k$, we have $\det(\mathscr{A}_k) \neq 0$. Since $\Cart_k$ is compatible with $W_S$, the conclusion now follows from Corollary \ref{cor:integer_representation}.
\end{proof}

\begin{proof}[Proof~of~Theorem~\ref{thm:surfacesubgroups}]
We first consider the case $n \geq 5$. For odd (respectively, even) such~$n$, the group $W_n$ generated by the reflections in the sides of a right-angled $n$-gon in the hyperbolic plane $\mathbb{H}^2$ possesses a subgroup of index $8$ (respectively, of index $4$) by which the quotient of $\mathbb{H}^2$  is a closed oriented surface of genus $n-3$ (resp., of genus $\frac{n-2}{2}$) \cite{Edmonds_Torsionfreesubgroups}. Moreover, the Coxeter diagrams of these hyperbolic reflection groups are not trees. The conclusion of Theorem \ref{thm:surfacesubgroups} thus follows from Proposition \ref{prop:thin}.

To obtain the statement for $n=4$, observe for instance that the integer matrix
$$
\left(
\begin{array}{rrrrr}
2 & -2 & 0 & 0 & -1 \\
-4 & 2 & -2 & 0 & 0 \\
0 & -4 & 2 & -2 & 0 \\
0 & 0 & -4 & 2 & -2 \\
-12 & 0 & 0 & -4 & 2 \\
\end{array}
\right)
$$
is of rank $4$, is a Cartan matrix compatible with $W_5$, and is not symmetrizable. We now conclude in this case by applying Corollary \ref{cor:integer_representation}.

Finally, the integer matrix
$$
\left(
\begin{array}{rrr}
2 & -1 & -1 \\
-2 & 2 & -1 \\
-1 & -1 & 2 \\
\end{array}
\right)
$$
is a full-rank Cartan matrix compatible with the $(3,3,4)$-triangle group and is not symmetrizable. Since this triangle group contains a finite-index genus-$2$ surface subgroup, the $n=3$ case now follows again by applying Corollary \ref{cor:integer_representation}. (This surface subgroup of $\mathrm{SL}_3(\mathbb{Z})$ was discovered by Kac--Vinberg \cite{vinberg1967quasi} and also appeared in subsequent work of Long--Reid--Thistlethwaite \cite{long2011zariski}.)
\end{proof}

\begin{comment}
\begin{theorem}
Let $W_S$ be an irreducible right-angled Coxeter group of rank $N$, and let $\rho \in \mathrm{Hom}^{\mathrm{ref}}(W_S,\mathrm{SL}^{\pm}_n(\mathbb{R}))$. Suppose that $\rho$ is irreducible and $\rho(W_S)$ is a Zariski dense subgroup of $\mathrm{SL}^{\pm}_n(\mathbb{Z})$. If $N \geq 3n+1$, then there exists an irreducible representation $\rho' \in \mathrm{Hom}^{\mathrm{ref}}(W_S,\mathrm{SL}^{\pm}_{n+1}(\mathbb{R}))$ such that $\rho'(W_S)$ is a Zariski dense subgroup of $\mathrm{SL}^{\pm}_{n+1}(\mathbb{Z})$.
\end{theorem}
\end{comment}

The following proposition provides the inductive step in the proof of Theorem \ref{thm:virtually-Zariski-dense}.

\begin{proposition}\label{prop:next-dimensions-achieved}
Let $W_S$ be an irreducible right-angled Coxeter group of rank $N$ and suppose one has a Zariski-dense representation $\rho:W_S\to\mathrm{SL}_n^\pm(\mathbb{Z})$ of $W_S$ as a reflection group, where $N \geq 3n+1$. Then there exists a Zariski-dense representation $\rho':W_S\to\mathrm{SL}_{n+1}^\pm(\mathbb{Z})$ of $W_S$ as a reflection group.
\end{proposition}

\begin{proof}
Let $S = \{s_1, \dotsc, s_N\}$. For each $i = 1, \dotsc, N$, write $\rho(s_i)=\mathrm{Id}-\alpha_i\otimes v_i$ where $\alpha_i \in ({\mathbb{R}^n})^*$ and $v_i \in \mathbb{R}^n$ satisfy $\alpha_i(v_i)=2$. Let $\mathcal{A} = \left(\mathcal{A}_{ij}\right)_{1 \leq i,j \leq N} = \left( \alpha_i(v_j) \right)_{1 \leq i,j \leq N}$ be the Cartan matrix for $\rho$. Since $\rho$ is irreducible, the intersection of the kernels of $\alpha_1, \dotsc, \alpha_N$ equals $\{ 0 \}$ and the span of~$v_1, \dotsc, v_N$ equals $\mathbb{R}^n$. So, there exist $i_1, \dotsc, i_n, j_1, \dotsc, j_n \in \{ 1, \dotsc, N\}$ such that the intersection of the kernels of  $\alpha_{i_1}, \dotsc, \alpha_{i_n}$ equals $\{ 0 \}$ and the span of $v_{j_1}, \dotsc, v_{j_n}$ equals $\mathbb{R}^n$. Define
$$ K := \{ 1, \dotsc, N \} \smallsetminus \left( \{ i_1, \dotsc, i_n \} \cup \{ j_1, \dotsc, j_n \} \right).$$
We observe that there exist distinct $i_0,j_0\in K$ such that $m_{i_0,j_0} = \infty$. Indeed, suppose for a contradiction that $m_{ij} = 2$ for all distinct $i,j \in K$. Then the submatrix $\mathcal{A}_K$ of $\mathcal{A}$ is the diagonal matrix $2\,\mathrm{Id}$, and hence 
$$ \mathrm{rank}(\mathcal{A}) \geq \mathrm{rank}(\mathcal{A}_K) \geq N - 2n \geq n+1.$$ Since $\mathrm{rank}(\mathcal{A}) = n$, we obtain a contradiction.

For each $i = 1, \dotsc, N $ and $t \in \mathbb{N}$, define $\alpha^t_i \in ({\mathbb{R}^{n+1}})^*$ and $v^t_i \in \mathbb{R}^{n+1}$ as follows: 
\[
\alpha^t_i  = \begin{cases} 
(\alpha_{i_0}, \sqrt{t}) & \textrm{ if } i = i_0 \\ 
(\alpha_i, 0) & \textrm{ otherwise,} 
\end{cases}
\quad\quad \textrm{and} \quad\quad
v^t_i  = \begin{cases} 
(v_{j_0}, \sqrt{t} \mathcal{A}_{i_0 j_0} ) & \textrm{ if } i = j_0 \\ 
(v_i, 0) & \textrm{ otherwise} 
\end{cases}
\]where we view $\mathbb{R}^{n+1}$ as $\mathbb{R}^n\times\mathbb{R}$.
Define $\Cart_t=\left( \alpha^t_i(v^t_j) \right)_{1 \leq i, j \leq N}$, the $(i,j)^\text{th}$ entry of $\Cart_t$ is given by \[
\alpha^t_i(v^t_j)  = \begin{cases} 
	(1+t)\mathcal{A}_{i_0 j_0} & \textrm{ if } (i,j) = (i_0,j_0) \\ 
	\mathcal{A}_{i j} & \textrm{ otherwise.} 
\end{cases}
\]
Hence, $\Cart_t$ is a Cartan matrix compatible with $W_S$. For $t\neq 0$, the intersection of the kernels of $\alpha^t_1, \dotsc, \alpha^t_N$ equals $\{ 0 \}$ and the span of $v^t_1, \dotsc, v^t_N$ equals $\mathbb{R}^{n+1}$ so that $\Cart_t$ is of rank $n+1$.
Since~$\rho(W_S)$ is a subgroup of $\mathrm{SL}^{\pm}_n(\mathbb{Z})$, all the cyclic products of $\mathcal{A}$ are integers by Lemma \ref{lem:cyclic-products-integers}. So, for each $t \in \mathbb{N}$, all the cyclic products of $\mathcal{A}_t$ are also integers.

We claim that for all $t \in \mathbb{N}$, except possibly one value, $\mathcal{A}_t$ is not symmetrizable. Indeed, we need to consider one of the following two cases: either there exists a nonzero cyclic product of length $\geq 3$
$$ \mathcal{A}_{k_1 k_2} \mathcal{A}_{k_2 k_3} \dotsm \mathcal{A}_{k_{\ell} k_1} $$
with $(i_0,j_0) = (k_1,k_2)$, or no such cyclic product exists. In the first case, for all $t\in\mathbb{N}$ except possibly one value, $$(\mathcal{A}_t)_{k_1 k_2} (\mathcal{A}_t)_{k_2 k_3} \dotsm (\mathcal{A}_t)_{k_\ell k_1}\neq(\mathcal{A}_t)_{k_1 k_{\ell}} (\mathcal{A}_t)_{k_{\ell} k_{\ell-1}} \dotsm (\mathcal{A}_t)_{k_2 k_1}$$ since the cyclic product on the left is a degree-one polynomial in $t$ and the one on the right is constant as $t$ varies. Hence for all $t\in\mathbb{N}$ except possibly one value, $\Cart_t$ is not symmetrizable. In the second case, every nonzero cyclic product 
$$ (\mathcal{A}_t)_{k_1 k_2} (\mathcal{A}_t)_{k_2 k_3} \dotsm (\mathcal{A}_t)_{k_{\ell} k_1} $$
of length $\geq 3$ is equal to $ \mathcal{A}_{k_1 k_2} \mathcal{A}_{k_2 k_3} \dotsm \mathcal{A}_{k_{\ell} k_1}$. But, since $\rho$ is Zariski-dense, the Cartan matrix~$\mathcal{A}$ is not symmetrizable. By Proposition \ref{prop:cyclic_products_equivalence}, this implies the existence of a nonzero cyclic product 
$$ \mathcal{A}_{k_1 k_2} \mathcal{A}_{k_2 k_3} \dotsm \mathcal{A}_{k_{\ell} k_1} \neq \mathcal{A}_{k_1 k_{\ell}} \mathcal{A}_{k_{\ell} k_{\ell-1}} \dotsm \mathcal{A}_{k_{2} k_1}$$ 
of $\mathcal{A}$, and such a cyclic product must have length $\geq 3$. So, $\Cart_t$ is not symmetrizable for any~$t\in\mathbb{N}$. 

The conclusion now follows from Corollary \ref{cor:integer_representation}.
\end{proof}

%We claim that for all $t\in\mathbb{N}$, except possibly one value, $\Cart_t$ is not symmetrizable. Observe that $\Cart_t$ and $\Cart$ are not equivalent since they do not have the same rank. \textcolor{blue}{Since $\Cart$ and $\Cart_t$ are compatible with $W_S$, it follows from Proposition \ref{prop:cyclic_products_equivalence} that there exist $l\geq 3$ and distinct indices $k_1,k_2,\dots,k_\ell$ such that
% $$(\mathcal{A}_t)_{k_1 k_2} (\mathcal{A}_t)_{k_2 k_3} \dotsm (\mathcal{A}_t)_{k_\ell k_1}\neq\mathcal{A}_{k_1 k_2} \mathcal{A}_{k_2 k_3} \dotsm \mathcal{A}_{k_{\ell} k_1}$$ with $(i_0,j_0)=(k_1, k_2)$.} Observe that the cyclic product on the left is a nonzero degree-one polynomial in $t$. In particular, for all $t\in\mathbb{N}$ except possibly one value, $$(\mathcal{A}_t)_{k_1 k_2} (\mathcal{A}_t)_{k_2 k_3} \dotsm (\mathcal{A}_t)_{k_\ell k_1}\neq(\mathcal{A}_t)_{k_1 k_{\ell}} (\mathcal{A}_t)_{k_{\ell} k_{\ell-1}} \dotsm (\mathcal{A}_t)_{k_2 k_1}$$ since the cyclic product on the right is constant as $t$ varies. Hence for all $t\in\mathbb{N}$ except possibly one value, $\Cart_t$ is not symmetrizable. The conclusion now follows from Corollary \ref{cor:integer_representation}.

The following two lemmas will allow us to pass to convenient finite-index reflection subgroups during the proof of Theorem \ref{thm:virtually-Zariski-dense}.

\begin{lemma}\label{lem:NotATree}
	Let $W_S$ be an irreducible right-angled Coxeter group with $|S| \geq 3$ whose Coxeter diagram is a tree. Then there exists an index-$2$ reflection subgroup $W_{S'}$ of $W_S$ with $|S'| = |S|$ whose Coxeter diagram contains a triangle.
\end{lemma}

\begin{proof}
Let $s_1, s_2, \dotsc, s_n$ be an enumeration of $S$, where $S$ is identified with the vertex set of the Coxeter diagram $\mathcal{G}_{W_S}$ of $W_S$. Since $\mathcal{G}_{W_S}$ is a tree, we may assume that $s_1$ is a leaf of~$\mathcal{G}_{W_S}$, i.e., a vertex of degree $1$, and that $s_2$ is the unique vertex adjacent to $s_1$. Furthermore, since $n \geq 3$ and $\mathcal{G}_{W_S}$ is connected, we may also assume that $s_3$ is adjacent to $s_2$. Note that $s_1$ commutes with $s_i$ for each $i \geq 3$, and hence $s_1 s_i s_1^{-1} = s_i$. The subgroup generated by $$S' = \{ s'_2, s_2, s_3, \dotsc, s_n\}, \textrm{ where } s'_2 = s_1 s_2 s_1^{-1},$$ is a subgroup of index $2$ in $W_{S}$ which is itself a Coxeter group $W_{S'}$, and the Coxeter diagram of $W_{S'}$ contains the triangle $\{ s'_2, s_2, s_3 \}$. 
\end{proof}

\begin{lemma}\label{lem:larger_rank}
	Let $W_S$ be an irreducible right-angled Coxeter group with $|S| \geq 3$. Then there is an index-$2$ reflection subgroup $W_{S'}$ of $W_S$ with $|S'| > |S|$. 
\end{lemma}

\begin{proof}
Let $s_1, s_2, \dotsc, s_n$ be an enumeration of $S$, where $S$ is identified with the vertex set of the Coxeter diagram $\mathcal{G}_{W_S}$ of $W_S$. Since $\mathcal{G}_{W_S}$ is connected and $n \geq 3$, we may assume that $s_1$ is adjacent precisely to $s_2, \ldots, s_m$ with $m \geq 3$. The subgroup generated by $$S' = \{ s'_2, \ldots, s'_m, s_2, s_3, \dotsc, s_n\}, \textrm{ where } s'_i = s_1 s_i s_1^{-1},$$ is a subgroup of index $2$ in $W_{S}$ which is itself a Coxeter group $W_{S'}$ of rank $n+m-2$.
\end{proof}

Arguing by induction on $n$, the following proposition is immediate from Lemma~\ref{lem:larger_rank} and Proposition~\ref{prop:next-dimensions-achieved}.

\begin{proposition}\label{prop:higher_dimensions}
Let $W_S$ be an irreducible right-angled Coxeter group of rank $\geq 3$. Suppose one has a Zariski-dense representation $\rho : W_S \rightarrow \SL_{m}^\pm (\Z)$ of $W_S$ as a reflection group. Then for every $n \geq m$, there exists a finite-index reflection subgroup $\Gamma_n$ of $W_S$ and a Zariski-dense representation $\rho_n : \Gamma_n \rightarrow \SL_{n}^\pm (\Z)$ of $\Gamma_n$ as a reflection group.
\end{proposition}

\begin{proof}[Proof of Theorem~\ref{thm:virtually-Zariski-dense}]
\begin{comment}
The proof proceeds by induction on $n$. First, if the Coxeter diagram of $W_S$ is not a tree then by Proposition~\ref{prop:thin}, there exists a representation of $W_S$ as a reflection group which is Zariski-dense in $\SL_N (\Z)$, where $N= |S|$.

If the Coxeter diagram of $W_S$ is a tree then by Lemma~\ref{lem:NotATree} there exists finite-index reflection group $W_{S'}$ such that the Coxeter diagram of $W_{S'}$ is not a tree. Now by Proposition~\ref{prop:thin}, there exists a representation of $W_{S'}$ as a reflection subgroup which is Zariski-dense in $\SL_N (\Z)$, where $N= |S'|$.

Finally, assume the existence of a finite-index reflection subgroup $\Gamma_{n}$ and a representation $\Gamma_{n}$ as a reflection group which is Zariski-dense in $\SL_{n}^\pm (\Z)$. Lemma~\ref{lem:NotATree} guarantees the existence of a finite-index reflection subgroup $\Gamma_{n+1}$ of $\Gamma_n$ whose Coxeter diagram is not a tree and whose rank is bigger than $3n+1$. Proposition~\ref{prop:next-dimensions-achieved} shows that there exists a representation of $\Gamma_{n+1}$ as a reflection group which is Zariski-dense in $\SL_{n+1} (\Z)$.
\end{comment}
By Proposition \ref{prop:thin}, if the Coxeter diagram of $W$ is not a tree, then there is a Zariski-dense representation $W \rightarrow \mathrm{SL}_N^\pm(\mathbb{Z})$ of $W$ as a reflection group. Otherwise, by Lemma \ref{lem:NotATree}, there is an index-$2$ reflection subgroup $\Gamma_N$ of $W$ of rank $N$ whose Coxeter diagram is not a tree. Hence, again by Proposition  \ref{prop:thin}, we obtain a Zariski-dense representation $\Gamma_N \rightarrow \mathrm{SL}_N^\pm(\mathbb{Z})$ of $\Gamma_N$ as a reflection group. We now conclude using Proposition \ref{prop:higher_dimensions}.
%We argue by induction on $n$. The base case is covered by Lemma \ref{lem:NotATree} and Proposition~\ref{prop:thin}. The inductive step is the content of the following proposition, which is immediate from Lemma~\ref{lem:larger_rank} and Proposition~\ref{prop:next-dimensions-achieved}.
\end{proof}

\subsection{Non-right-angled thin reflection groups} The following theorem, which will be applied in \S\ref{sec:thin-hyperbolic-manifold-groups}, generalizes Proposition \ref{prop:thin} to certain Coxeter groups beyond the right-angled setting.

\begin{theorem}\label{thm:CartanMatrixOfFullRank}
Let $W_S$ be an irreducible Coxeter group of rank $N$ and with Coxeter diagram~$\mathcal{G}_{W_S}$. Assume $W_S$ satisfies the following:
\begin{enumerate}
\item\label{item:mij} $m_{ij} \in \{ 2, 3, 4, 6, \infty \}$ for all $s_i \neq s_j \in S$;

\item\label{item:cycle} $\mathcal{G}_{W_S}$ contains a cycle $\mathcal{C}$ which is not of the form $\widetilde{A}_k$ for any $k \geq 2$;

\item\label{item:delete_edges} there exists a subset $T = \{ s_{\ell_1}, s_{\ell'_1}, s_{\ell_2}, s_{\ell'_2}, \dotsc, s_{\ell_q}, s_{\ell'_q} \}$ of $S$ such that 

\begin{itemize}
\item $m_{\ell_p\ell'_p} = \infty$ for all $p = 1, \dotsc, q$; 

\item every irreducible factor of $W_{U}$, where $U = S \smallsetminus T$, is either spherical or quasi-Lann\'{e}r.
\end{itemize}
\end{enumerate}
Then there is a Zariski-dense representation $W_S\to\SL^{\pm}_{N}(\mathbb{Z})$ of $W_S$ as a reflection group.
\end{theorem}

\begin{proof}
We may assume that $S = \{s_1, \dotsc, s_N\}$, where the edges $\overline{s_{1}s_{2}}$, $\overline{s_{2}s_{3}}$, $\dotsc$, $\overline{s_{r-1}s_{r}}$, and $\overline{s_{r}s_{1}}$ form the cycle $\mathcal{C}$. Let $\mathcal{T}$ be the path in $\mathcal{C}$ with edges $\overline{s_{1}s_{2}}$, $\overline{s_{2}s_{3}}$, $\dotsc$, $\overline{s_{r-1}s_{r}}$. Let $\mathcal{A}$ be the $N \times N$ real matrix whose $(i,j)^\text{th}$ entry $\mathcal{A}_{i j}$ is given by
\[
\mathcal{A}_{i j} = \begin{cases} 2 & \textrm{ if } i=j \\ 
0 & \textrm{ if } i \neq j \text{ and $s_i$ and $s_j$ are not adjacent in $\mathcal{G}_{W_S}$} \\ 
-1 & \textrm{ if } i < j \text{ and the edge $\overline{s_is_j}$ lies in $\mathcal{T}$}  \\
-4\cos^{2}\left(\tfrac{\pi}{m_{ij}}\right) & \textrm{ if } i > j \text{ and the edge $\overline{s_is_j}$ lies in $\mathcal{T}$} \\
-1 & \textrm{ if } i > j \text{ and the edge $\overline{s_is_j}$ does not lie in $\mathcal{T}$} \\ 
-4\cos^{2}\left(\tfrac{\pi}{m_{ij}}\right) & \textrm{ if } i < j \text{ and the edge $\overline{s_is_j}$ does not lie in $\mathcal{T}$}.  \end{cases}
\]
The values $4\cos^{2}\left(\tfrac{\pi}{m_{ij}}\right)$ are integers by assumption (\ref{item:mij}). Since $\mathcal{C}$ is not of the form $\widetilde{A}_k$ by assumption (\ref{item:cycle}),
$$| \mathcal{A}_{s_1 s_r}\mathcal{A}_{s_r s_{r-1}}\dotsm \mathcal{A}_{s_3 s_2}\mathcal{A}_{s_2 s_1}| = 4\cos^{2}\left(\tfrac{\pi}{m_{s_1 s_r}}\right) \dotsm 4\cos^{2}\left(\tfrac{\pi}{m_{s_2 s_1}}\right) > 1 = | \mathcal{A}_{s_1 s_2}\mathcal{A}_{s_2 s_3}\dotsm \mathcal{A}_{s_{r-1} s_r}\mathcal{A}_{s_r s_1}|.$$
Consequently, $\mathcal{A}$ is not symmetrizable by Proposition \ref{prop:cyclic_products_equivalence}. In other words, $W_S$ admits an integral Cartan matrix which is not symmetrizable. However, in general, $\mathcal{A}$ might not be of full rank. We therefore modify the Cartan matrix $\mathcal{A}$ as follows. Consider the $N \times N$ matrix~$\mathscr{A}_t$ whose $(i,j)^\text{th}$ entry $(\mathscr{A}_t)_{i j}$ is given by
\[
(\mathscr{A}_t)_{i j} = \begin{cases} 
 t \mathscr{A}_{i j} & \text{ if  $\{ i, j \} = \{ \ell_p, \ell'_p \}$ for some $p \in \{ 1, \dotsc, q \}$} \\ \mathscr{A}_{i j} & \text{otherwise}.  
\end{cases}
\]
The matrix $\mathscr{A}_t$ is not symmetrizable for $t \neq 0$ since $\mathscr{A}$ is not symmetrizable. The determinant of $\mathscr{A}_t$ is a polynomial of degree $2q$ with leading coefficient
$$ \mathscr{A}_{\ell_1 \ell'_1} \mathscr{A}_{\ell'_1 \ell_1} \dotsm \mathscr{A}_{\ell_q \ell'_q} \mathscr{A}_{\ell'_q \ell_q} \det(\mathscr{A}_{U}).$$
Since each irreducible factor of $W_U$ is spherical or quasi-Lann\'{e}r, we have $\det(\mathscr{A}_{U}) \neq 0$, which implies that $\det(\mathscr{A}_{t})$ is not the zero polynomial. Thus, for a sufficiently large positive integer~$t_0$, we have $\det(\mathscr{A}_{t_0}) \neq 0$. Note also that $W_S$ is large since $\mathcal{G}_{W_S}$ contains a cycle  which is not of the form $\widetilde{A}_k$ for any $k \geq 2$. Thus, the irreducible Coxeter group $W_S$ and the matrix $\mathscr{A}_{t_0}$ satisfy the assumptions of Corollary \ref{cor:integer_representation}, which concludes the proof.
\end{proof}

\begin{remark}
In the case that $W_S$ is moreover Gromov-hyperbolic, it follows from \cite[Cor.~1.18]{danciger2023convex} that one can arrange for the output Zariski-dense representation $W_S \rightarrow \mathrm{SL}_N^\pm(\mathbb{Z})$ in the statement of Theorem \ref{thm:CartanMatrixOfFullRank} to be $P_1$-Anosov by replacing each appearance of $-4$ as an entry of the Cartan matrix $\mathcal{A}$ in the above proof with, say, $-5$.
\end{remark}

\section{Thin hyperbolic manifold groups}\label{sec:thin-hyperbolic-manifold-groups}

In this section, we apply Theorem \ref{thm:CartanMatrixOfFullRank} to prove Theorem \ref{thm:hyperbolic_manifolds}, which consists of the following three propositions. 

\begin{proposition}\label{prop:application-three-manifold}
	There exists a closed hyperbolic $3$-manifold $M$ such that for each $n \geq 4$, there is a finite-index subgroup $\Gamma_n<\pi_1(M)$ that embeds Zariski-densely in $\mathrm{SL}_n(\mathbb{Z})$.
\end{proposition}

\begin{proof}
Let $k$ be an integer $\geq 1$. By Andreev's theorem \cite{Andreev}, there exist compact hyperbolic Coxeter $3$-polytopes $P_k \subset \mathbb{H}^3$ with $k+4$ facets as in Figure \ref{fig:three-dimension}. 
\newcommand{\scalee}{1.0}
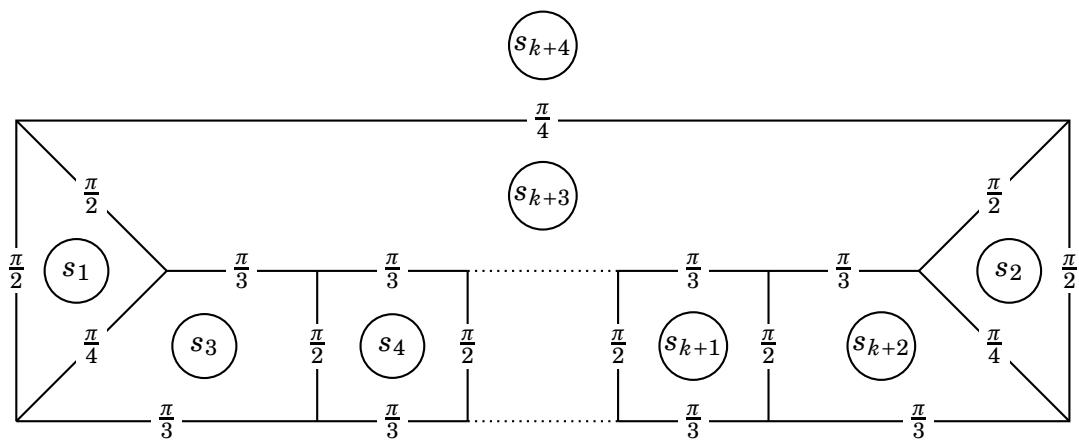
\begin{figure}[ht!]
\centering
\subfloat{
\begin{tikzpicture}[thick,scale=\scalee, every node/.style={transform shape}]
\draw (0,0) -- (6,0);
\draw[dotted] (6,0) -- (8,0);
\draw (8,0) -- (14,0) -- (14,4) -- (0,4) -- (0,0);
\draw (0,0) -- (2,2) -- (6,2); 
\draw[dotted] (6,2) -- (8,2); 
\draw (8,2) -- (12,2) -- (14,0);
\draw (0,4) -- (2,2);
\draw (14,4) -- (12,2);
\draw (4,2) -- (4,0);
\draw (6,2) -- (6,0);
\draw (8,2) -- (8,0);
\draw (10,2) -- (10,0);
\node[draw,circle, inner sep=4pt, minimum size=2pt] (1) at (0.8,2) {$s_1$};
\node[draw,circle, inner sep=4pt, minimum size=2pt] (2) at (2.5,1) {$s_3$};
\node[draw,circle, inner sep=4pt, minimum size=2pt] (3) at (5,1) {$s_4$};
\node[draw,circle, inner sep=1pt, minimum size=2pt] (n) at (9,1) {$s_{k+1}$};
\node[draw,circle, inner sep=1pt, minimum size=2pt] (n+1) at (11.5,1) {$s_{k+2}$};
\node[draw,circle, inner sep=4pt, minimum size=2pt] (n+2) at (13.2,2) {$s_2$};
\node[draw,circle, inner sep=1pt, minimum size=2pt] (n+3) at (7,3) {$s_{k+3}$};
\node[draw,circle, inner sep=1pt, minimum size=2pt] (n+4) at (7,5) {$s_{k+4}$};
\node[fill=white, inner sep=2pt, minimum size=2pt] (s2b) at (2,0) {$\tfrac{\pi}{3}$};
\node[fill=white, inner sep=2pt, minimum size=2pt] (s3b) at (5,0) {$\tfrac{\pi}{3}$};
\node[fill=white, inner sep=2pt, minimum size=3pt] (snb) at (9,0) {$\tfrac{\pi}{3}$};
\node[fill=white, inner sep=2pt, minimum size=3pt] (sn+1b) at (12,0) {$\tfrac{\pi}{3}$};
\node[fill=white, inner sep=2pt, minimum size=2pt] (s2t) at (3,2) {$\tfrac{\pi}{3}$};
\node[fill=white, inner sep=2pt, minimum size=2pt] (s3t) at (5,2) {$\tfrac{\pi}{3}$};
\node[fill=white, inner sep=2pt, minimum size=3pt] (snt) at (9,2) {$\tfrac{\pi}{3}$};
\node[fill=white, inner sep=2pt, minimum size=3pt] (sn+1t) at (11,2) {$\tfrac{\pi}{3}$};
\node[fill=white, inner sep=2pt, minimum size=3pt] (s2l) at (1,1) {$\tfrac{\pi}{4}$};
\node[fill=white, inner sep=2pt, minimum size=3pt] (s2r) at (4,1) {$\tfrac{\pi}{2}$};
\node[fill=white, inner sep=2pt, minimum size=3pt] (s3r) at (6,1) {$\tfrac{\pi}{2}$};
\node[fill=white, inner sep=2pt, minimum size=3pt] (snl) at (8,1) {$\tfrac{\pi}{2}$};
\node[fill=white, inner sep=2pt, minimum size=3pt] (snr) at (10,1) {$\tfrac{\pi}{2}$};
\node[fill=white, inner sep=2pt, minimum size=3pt] (sn+1r) at (13,1) {$\tfrac{\pi}{4}$};
\node[fill=white, inner sep=2pt, minimum size=3pt] (sn+3t) at (7,4) {$\tfrac{\pi}{4}$};
\node[fill=white, inner sep=2pt, minimum size=3pt] (s1l) at (0,2) {$\tfrac{\pi}{2}$};
\node[fill=white, inner sep=2pt, minimum size=3pt] (s1r) at (1,3) {$\tfrac{\pi}{2}$};
\node[fill=white, inner sep=2pt, minimum size=3pt] (sn+2l) at (13,3) {$\tfrac{\pi}{2}$};
\node[fill=white, inner sep=2pt, minimum size=3pt] (sn+2r) at (14,2) {$\tfrac{\pi}{2}$};
\end{tikzpicture}
}
\caption{A family of compact hyperbolic Coxeter $3$-polytopes $P_k$}\label{fig:three-dimension}
\end{figure}

Let $W_k$ be the reflection group in $\mathrm{Isom}(\mathbb{H}^3)$ generated by the set of reflections $$S = \{s_{1}, s_{2}, \dotsc, s_{k+3}, s_{k+4}\}$$ in the facets of $P_k$. Note that \( P_1 \) is a triangular prism and, for each \( k \geq 2 \), the Coxeter polytope \( P_k \) is isometric to 
\[
P_1 \;\cup\; w_2(P_1) \;\cup\; w_3(P_1) \;\cup\; \dotsm \;\cup\; w_k(P_1),
\]
where \( (w_2, w_3, w_4, w_5, \dotsc) = (s_2, s_1 s_2, s_2 s_1 s_2, s_1 s_2 s_1 s_2, \dotsc) \).
So $W_k$ is a subgroup of $W_1$ of index~$k$. From now on, we consider $W_k$ as an abstract Coxeter group. We show that $W_k$ satisfies all the assumptions of Theorem \ref{thm:CartanMatrixOfFullRank}. First, the dihedral angles in $P_k$ are either~$\tfrac{\pi}{2}$, $\tfrac{\pi}{3}$, or $\tfrac{\pi}{4}$, so that assumption (\ref{item:mij}) is satisfied. Second, the edges $\overline{s_{3} s_{k+3}}$, $\overline{s_{k+3} s_{k+4}}$ and $\overline{s_{k+4} s_{3}}$ form a cycle which is not of the form $\widetilde{A}_\ell$. For assumption (\ref{item:delete_edges}), we consider two cases: when $k$ is odd or when $k$ is even. In the case where $k$ is odd, if we set $T = \{ s_1, s_2 \} \cup \{s_3, s_{k+2}, s_4, s_{k+1}, \dotsc, s_{(k+5)/2-1}, s_{(k+5)/2+1}\}$, then \[m_{1,2} = m_{3,k+2} = m_{4,k+1} = \dotsm = m_{(k+5)/2-1,(k+5)/2+1} = \infty\]and $(W_{k})_U$, where $U = S \smallsetminus T = \{ s_{(k+5)/2}, s_{k+3}, s_{k+4}\}$, is a Lann\'{e}r Coxeter group. In the case where $k$ is even, if we instead set $$T = \{ s_1, s_{k+2}\} \cup \{s_3, s_{k+1}, s_4, s_{k}, \dotsc, s_{(k+4)/2-1}, s_{(k+4)/2+1}\} \cup \{s_{2}, s_{(k+4)/2}\},$$ then $m_{1,k+2} = m_{3,k+1} = m_{4,k} = \dotsm = m_{(k+4)/2-1,(k+4)/2+1} = m_{2,(k+4)/2} = \infty $ and $(W_k)_U$, where $U = S \smallsetminus T = \{ s_{k+3}, s_{k+4}\}$, is a spherical Coxeter group. Therefore, Theorem \ref{thm:CartanMatrixOfFullRank} guarantees the existence of a Zariski-dense representation $\rho_k:W_k\to\SL_{k+4}^{\pm}(\mathbb{Z})$ of $W_k$ as a reflection group.
Choose any finite-index orientation-preserving torsion-free subgroup $\Gamma$ of $W_1$. Then $\Gamma$ can be identified with the fundamental group $\pi_1(M)$ of a closed hyperbolic $3$-manifold $M$, and for each $k \geq 1$, the finite-index subgroup $\Gamma \cap W_k$ of $\Gamma$ embeds Zariski-densely in $\mathrm{SL}_{k+4} (\mathbb{Z})$ via $\rho_k$. 

Finally, we observe that the following integral Cartan matrix of rank $4$ is not symmetrizable and is compatible with $W_1$: 
$$
\left(
\begin{array}{rrrrr}
2 & -3 & -1 & 0 & 0 \\
-8 & 2 & -1 & 0 & 0 \\
-2 & -2 & 2 & -1 & -1 \\
0 & 0 & -1 & 2 & -1 \\
0 & 0 & -1 & -2 & 2 \\
\end{array}
\right).
$$
It thus follows from Corollary \ref{cor:integer_representation}, that there is a Zariski-dense representation $ W_1 \rightarrow \mathrm{SL}^{\pm}_{4} (\mathbb{Z})$ of $W_1$ as a reflection group which (necessarily) embeds $\Gamma$ as a Zariski-dense subgroup of $\mathrm{SL}_{4} (\mathbb{Z})$.
\end{proof}

\begin{remark}
Proposition \ref{prop:application-three-manifold} also admits an alternative argument that proceeds roughly as follows. There is a right-angled compact hyperbolic $3$-polytope, known as the L\"{o}bell polytope~$L_6$, such that the discrete subgroup of $\mathrm{O}_{3,1} (\R)$ generated by the reflections in the facets of $L_6$ is conjugate  within $\mathrm{SL}^{\pm}_4(\mathbb{R})$ to a subgroup $\Gamma$ of $\mathrm{SL}^{\pm}_4(\mathbb{Z})$; see \cite{bogachevdouba}. One can now double this polytope along any of its facets $F$ to obtain an index-$2$ reflection subgroup $\Gamma' < \Gamma$, and then ``bend $\Gamma'$ along $F$'' via a well-chosen element of $\mathrm{SL}_4(\mathbb{R})$ such that $\Gamma'$ becomes Zariski-dense in $\mathrm{SL}^{\pm}_4(\mathbb{R})$ but nevertheless remains within $\mathrm{SL}^{\pm}_4(\mathbb{Z})$. One now concludes using Proposition \ref{prop:higher_dimensions}.
\end{remark}

\begin{proposition}\label{prop:application-four-manifold}
	There exists a complete hyperbolic $4$-manifold $M$ of finite volume such that for each $n \geq 5$, there is a finite-index subgroup $\Gamma_n<\pi_1(M)$ that embeds Zariski-densely in $\mathrm{SL}_n(\mathbb{Z})$.
\end{proposition}

\begin{proof}
In the proof of Proposition \ref{prop:application-three-manifold}, we introduced the compact hyperbolic Coxeter 3-polytope $P_1$ which is combinatorially the Cartesian product of a triangle and an interval; see Figure \ref{fig:three-dimension}. The Coxeter diagram of $P_1$ is shown on the left of Figure \ref{fig:application-Coxeter-diagram}. 
\newcommand{\scalee}{0.9}
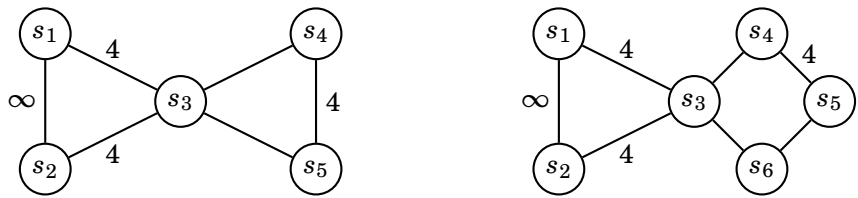
\begin{figure}[ht!]
\centering
\subfloat{
\begin{tikzpicture}[thick,scale=\scalee, every node/.style={transform shape}]
\node[draw,circle, inner sep=3pt, minimum size=2pt] (1) at (0,2) {$s_1$};
\node[draw,circle, inner sep=3pt, minimum size=2pt] (2) at (0,0) {$s_2$};
\node[draw,circle, inner sep=3pt, minimum size=2pt] (3) at (2,1) {$s_3$};
\node[draw,circle, inner sep=3pt, minimum size=2pt] (4) at (4,2) {$s_4$};
\node[draw,circle, inner sep=3pt, minimum size=2pt] (5) at (4,0) {$s_5$};
\draw (1)--(2) node[left, midway] {$\infty$};
\draw (1)--(3) node[above, midway] {$4$};
\draw (2)--(3) node[below, midway] {$4$};
\draw (3)--(4) node[right,midway] {};
\draw (3)--(5) node[right,midway] {};
\draw (4)--(5) node[right,midway] {$4$};
\end{tikzpicture}
}
\quad\quad\quad\quad\quad
\subfloat{
\begin{tikzpicture}[thick,scale=\scalee, every node/.style={transform shape}]
\node[draw,circle, inner sep=3pt, minimum size=2pt] (1) at (0,2) {$s_1$};
\node[draw,circle, inner sep=3pt, minimum size=2pt] (2) at (0,0) {$s_2$};
\node[draw,circle, inner sep=3pt, minimum size=2pt] (3) at (2,1) {$s_3$};
\node[draw,circle, inner sep=3pt, minimum size=2pt] (4) at (3,2) {$s_4$};
\node[draw,circle, inner sep=3pt, minimum size=2pt] (5) at (4,1) {$s_5$};
\node[draw,circle, inner sep=3pt, minimum size=2pt] (6) at (3,0) {$s_6$};
\node[] () at (3.7,1.7) {$4$};
\draw (1)--(2) node[left, midway] {$\infty$};
\draw (1)--(3) node[above, midway] {$4$};
\draw (2)--(3) node[below, midway] {$4$};
\draw (3)--(4) node[] {};
\draw (3)--(6) node[] {};
\draw (4)--(5) node[] {};
\draw (5)--(6) node[] {};
\end{tikzpicture}
}
\caption{The Coxeter diagrams of the hyperbolic $3$-polytope $P_1$ and the hyperbolic $4$-polytope $Q_1$}\label{fig:application-Coxeter-diagram}
\end{figure}

In analogy to the $3$-dimensional case, we consider a noncompact hyperbolic 4-polytope \( Q_1 \) of finite volume, whose Coxeter diagram is shown on the right of Figure \ref{fig:application-Coxeter-diagram}. This polytope is combinatorially the Cartesian product of a tetrahedron and an interval. For \( k \geq 2 \), let \( Q_k \) be the Coxeter polytope  
\[
Q_1 \;\cup\; w_2(Q_1) \;\cup\; w_3(Q_1) \;\cup\; \dotsm \;\cup\; w_k(Q_1),
\]
where \( (w_2, w_3, w_4, w_5, \dotsc) = (s_2, s_1 s_2, s_2 s_1 s_2, s_1 s_2 s_1 s_2, \dotsc) \). Let $W_k$ be the reflection group in $\mathrm{Isom}(\mathbb{H}^4)$ generated by the reflections in the facets of $Q_k$. Then $W_k$ is a subgroup of $W_1$ of index $k$. As in the proof of Proposition \ref{prop:application-three-manifold}, it can be shown that $W_k$ satisfies all the assumptions of Theorem \ref{thm:CartanMatrixOfFullRank}. 

To obtain a Zariski-dense representation $W_1 \rightarrow \mathrm{SL}_5^\pm(\mathbb{Z})$ of $W_1$ as a reflection group, observe that the following integral Cartan matrix of rank $5$ is not symmetrizable and is compatible with $W_1$: 
$$
\left(
\begin{array}{rrrrrr}
2 & -4 & -1 & 0 & 0 & 0 \\
-4 & 2 & -1 & 0 & 0 & 0 \\
-2 & -2 & 2 & -1 & 0 & -1 \\
0 & 0 & -1 & 2 & -1 & 0 \\
0 & 0 & 0 & -2 & 2 & -1 \\
0 & 0 & -1 & 0 & -1 & 2 \\
\end{array}
\right).
$$ 
The remainder of the proof is very similar to the proof of Proposition \ref{prop:application-three-manifold}.
\end{proof}

\begin{proposition}\label{prop:application-aspherical-manifold}
For every $p \geq 4$, there exists a closed aspherical manifold $M_p$ of dimension~$p$ such that for all $n\geq 2p$, there is a finite-index subgroup $\Gamma_n<\pi_1(M_p)$ that embeds Zariski-densely in $\mathrm{SL}_n(\mathbb{Z})$.
\end{proposition}

\begin{proof}
Let $p$ be an integer $\geq 4$, and let $W_1$ be the Coxeter group with Coxeter diagram as in Figure \ref{fig:application-Coxeter-diagram-cube}.
\newcommand{\scalee}{0.85}
\begin{figure}[ht!]
\centering
\subfloat{
\begin{tikzpicture}[thick,scale=\scalee, every node/.style={transform shape}]
\node[draw,circle, inner sep=6pt, minimum size=2pt] (1) at (0,2) {$s_1$};
\node[draw,circle, inner sep=6pt, minimum size=2pt] (2) at (0,0) {$s_2$};
\node[draw,circle, inner sep=6pt, minimum size=2pt] (3) at (2,1) {$s_3$};
\node[draw,circle, inner sep=6pt, minimum size=2pt] (4) at (4,1) {$s_4$};
\node[draw,circle, inner sep=6pt, minimum size=2pt] (5) at (6,1) {$s_5$};
\node[draw,circle, inner sep=6pt, minimum size=2pt] (6) at (8,1) {$s_6$};
\node[draw,circle, inner sep=0.5pt, minimum size=2pt] (7) at (11,1) {$s_{2p-2}$};
\node[draw,circle, inner sep=0.5pt, minimum size=2pt] (8) at (13,1) {$s_{2p-1}$};
\node[draw,circle, inner sep=4pt, minimum size=2pt] (9) at (15,1) {$s_{2p}$};
\node (h4) at (9,1) {};
\node (t4) at (10,1) {};
\draw (1)--(2) node[left, midway] {$\infty$};
\draw (1)--(3) node[above, midway] {$4$};
\draw (2)--(3) node[below, midway] {$4$};
\draw (3)--(4) node[above, midway] {$\infty$};
\draw (4)--(5) node[below, midway] {};
\draw (5)--(6) node[above, midway] {$\infty$};
\draw (6)--(h4) node[above, midway] {};
\draw[dotted] (h4)--(t4) node[below, midway] {};
\draw (t4)--(7) node[above, midway] {};
\draw (7)--(8) node[below, midway] {};
\draw (8)--(9) node[above, midway] {$\infty$};
\end{tikzpicture}
}
\caption{A family of Coxeter groups}\label{fig:application-Coxeter-diagram-cube}
\end{figure}
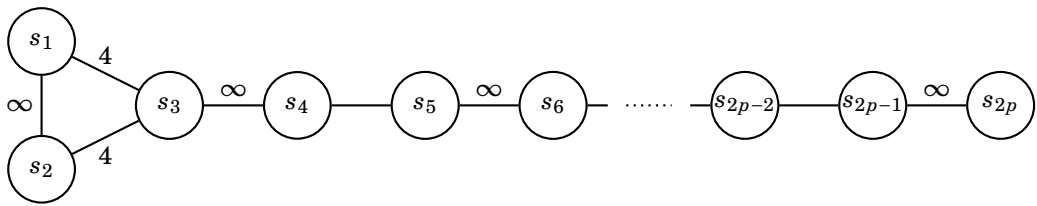

Let $\Sigma_1$ be the Davis complex of $W_1$ (see \cite[Ch.~7]{davis2008geometry}). Then the nerve of $W_1$ is isomorphic to the boundary complex of the dual polytope of the $p$-dimensional cube, hence $\Sigma_1$ is homeomorphic to $\mathbb{R}^p$. Let \( C_1 \) be a fundamental chamber for the action of $W_1$ on $\Sigma_1$. Then $C_1$ is the~$p$-dimensional cube. For \( k \geq 2 \), define \( C_k \) to be the union  
\[
C_1 \;\cup\; w_2(C_1) \;\cup\; w_3(C_1) \;\cup\; \dotsm \;\cup\; w_k(C_1),
\]
where  
$
(w_2, w_3, w_4, w_5, \dotsc) = (s_2, s_1 s_2, s_2 s_1 s_2, s_1 s_2 s_1 s_2, \dotsc).
$
Let \( W_k \) be the subgroup of \( W_1 \) generated by the reflections in the walls of \( C_k \). Then \( W_k \) is a subgroup of \( W_1 \) with index \( k \). As in the proof of Proposition \ref{prop:application-three-manifold}, the Coxeter group \( W_k \) satisfies all the conditions of Theorem~\ref{thm:CartanMatrixOfFullRank}. The remainder of the proof proceeds analogously to the proof of Proposition \ref{prop:application-three-manifold}.
\end{proof}

\section{New witnesses to incoherence of $\mathrm{SL}_n(\mathbb{Z})$}\label{sec:incoherence}

This section is dedicated to the proof of Theorem~\ref{thm:incoherence}. We begin with the following lemma, which is certainly well known. We include a proof for the convenience of the reader.

\begin{lemma}\label{lemma:ends}
	The fundamental group of a finite graph of groups all of whose vertex groups are one-ended and all of whose edge groups are infinite is one-ended.
\end{lemma}

\begin{proof}
	We proceed by induction on the number of edges in the graph of groups. In the absence of any edges, the statement trivially holds. Now suppose that for some $k \geq 0$, the fundamental group of any graph of groups as in the statement of Lemma \ref{lemma:ends} with $\leq k$ edges is one-ended, and let $\mathcal{G}$ be a graph of groups as in the statement of Lemma \ref{lemma:ends} with precisely $k+1$ edges.
	
	We use Stallings' characterization \cite{stallings1971} of a one-ended group. Supposing that $\Gamma:= \pi_1(\mathcal{G})$ acts by automorphisms on a tree $T$ without edge inversions and with edge stabilizers all finite, we will show that $\Gamma$ fixes a vertex in $T$. Stallings' theorem then implies that $\G$ is one-ended.
	
	Pick an edge $e$ of $\mathcal{G}$ and let $\Gamma_e < \Gamma$ be the corresponding edge group. Suppose first that~$e$ disconnects $\mathcal{G}$ into two components $\mathcal{G}_1$ and~$\mathcal{G}_2$. By the induction hypothesis, we have that~$\pi_1(\mathcal{G}_i)$ is one-ended and hence fixes a vertex $v_i$ of $T$ for $i=1,2$. Hence $\Gamma_e$ fixes the entire path between $v_1$ and $v_2$ since $\Gamma_e \subset \pi_1(\mathcal{G}_i)$ for $i=1,2$. If $v_1 \neq v_2$, this implies that $\Gamma_e$ fixes at least one edge of $T$, and hence that $\Gamma_e$ is finite, a contradiction. Thus $v_1=v_2$ is fixed by $\pi_1(\mathcal{G}_i)$ for $i=1,2$, and hence by $\Gamma$ since $\Gamma = \langle \pi_1(\mathcal{G}_1),  \pi_1(\mathcal{G}_2) \rangle$.
	
	Now suppose the complement $e$ in $\mathcal{G}$ is a connected graph $\mathcal{H}$, and let $\Delta = \pi_1(\mathcal{H}) < \Gamma$. Then we may view $\Gamma$ as an HNN extension $\Delta *_{\Gamma_e}$ with stable letter some element $t \in \Gamma$. By the induction hypothesis, we again have that $\Delta$ fixes a vertex $v$ of $T$, and hence $t \Delta t^{-1}$ fixes the vertex $tv$. We conclude that $\Gamma_e$ fixes the entire path between $v$ and $tv$ since $\Gamma_e \subset \Delta \cap t \Delta t^{-1}$. If $tv \neq v$, this implies that $\Gamma_e$ fixes at least one edge of $T$, and hence that $\Gamma_e$ is finite, a contradiction. Thus $tv=v$, and so $\Gamma$ fixes $v$ since $\Gamma = \langle \Delta, t \rangle$.
\end{proof}

\begin{proof}[Proof~of~Theorem~\ref{thm:incoherence}]
	We adapt an idea due to Bowditch and Mess \cite{MR1240944}. Let $W_S$ be the Coxeter group associated to a compact hyperbolic triangular Coxeter prism $P \subset \mathbb{H}^3$ possessing a facet $F$ orthogonal to all adjacent facets, and let $s \in S$ correspond to the facet $F$. We require moreover that $F$ be a $(p,q,r)$-triangle for $p, q, r \in \{3,4,6\}$. For each $m \geq~1$, let $W_m$ be the Coxeter group obtained from $W_S$ by adjoining to $S$ pairwise commuting involutions $s_1, \ldots, s_m$ such that $s_i$ commutes with an element of $s' \in S$ if $s$ commutes with $s'$ and otherwise shares no relation with $s'$; see Figure~\ref{fig:proof_for_one-ended} for an example. Then $W_m$ retracts onto the finite standard subgroup $\langle s, s_1, \ldots, s_m \rangle < W_m$. The kernel of this retraction is a reflection subgroup of $W_m$ that we may view as the fundamental group of an orbicomplex $R_m$ obtained by gluing $2^m$ copies of a compact hyperbolic reflection orbifold $O$ along a closed embedded totally geodesic hypersurface $\Sigma \subset O$. (The orbifold $O$ is obtained by doubling $P$ across $F$, and $\Sigma$ is precisely the ``forgotten'' facet $F$.)

\newcommand{\scalee}{0.9}
\begin{figure}[ht!]
\centering
\subfloat{
\begin{tikzpicture}[thick,scale=\scalee, every node/.style={transform shape}]
\node[draw,circle, inner sep=5pt, minimum size=2pt] (1) at (0,2) {$\phantom{s_!}$};
\node[draw,circle, inner sep=5pt, minimum size=2pt] (2) at (0,0) {$\phantom{s_!}$};
\node[draw,circle, inner sep=5pt, minimum size=2pt] (3) at (2,1) {$\phantom{s_!}$};
\node[draw,circle, inner sep=5pt, minimum size=2pt] (4) at (4,1) {$\phantom{s_!}$};
\node[draw,circle, inner sep=6.8pt, minimum size=2pt] (5) at (6,3) {$s$};
\node[draw,circle, inner sep=5pt, minimum size=2pt] (6) at (6,1.9) {$s_1$};
\node[draw,circle, inner sep=0.5pt, minimum size=2pt] (7) at (6,0.1) {$s_{m-1}$};
\node[draw,circle, inner sep=4pt, minimum size=2pt] (8) at (6,-1) {$s_{m}$};
\node (h5) at (5,1.3) {};
\node (t5) at (5,0.7) {};
\node (h6) at (6,1.3) {};
\node (t6) at (6,0.7) {};

\draw (1)--(2) node[left, midway] {$4$};
\draw (1)--(3) node[above, midway] {};
\draw (2)--(3) node[below, midway] {};
\draw (3)--(4) node[above, midway] {};
\draw (4)--(5) node[above, midway] {$\infty$};
\draw (4)--(6) node[above, midway] {$\infty$};
\draw (4)--(7) node[below, midway] {$\infty$};
\draw (4)--(8) node[below, midway] {$\infty$};

\draw[dotted] (h5)--(t5) node[above, midway] {};
\draw[dotted] (h6)--(t6) node[above, midway] {};
\end{tikzpicture}
}
\caption{A Coxeter group $W_m$ as in the proof of Theorem \ref{thm:incoherence}}\label{fig:proof_for_one-ended}
\end{figure}
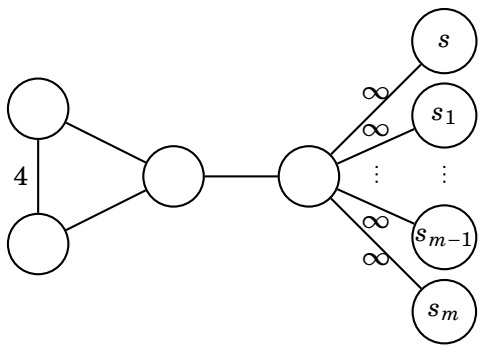

	Let
 $M$ be a finite cover of $O$ that topologically fibers over the circle (that such a cover exists for compact hyperbolic reflection $3$-orbifolds is due to Haglund--Wise \cite{haglund2010coxeter} and Agol \cite{agol2008criteria}). Then $R_m$ is finitely covered by a complex $C_m$ obtained by gluing $2^m$ copies of $M$ along a collection of disjoint closed embedded totally geodesic hypersurfaces (namely, the lifts of $\Sigma$). A fibration of $M$ gives rise to an infinite cyclic cover $\widehat{C_m}$ of $C_m$ whose fundamental group $\Gamma_m$ is finitely generated; indeed, we may view $\Gamma_m$ as the fundamental group of a finite graph of groups whose vertex groups are closed hyperbolic surface groups (and whose edge groups are infinite-rank free groups). Since $C_m$ has nonzero Euler characteristic, the homology of $\widehat{C_m}$ is nevertheless infinite-dimensional; see Milnor \cite{MR0242163}. The latter implies that $\Gamma_m$ is not finitely presented, since $\widehat{C_m}$ has the homotopy type of a $2$-complex. Finally, note that $\Gamma_m$ is one-ended by Lemma \ref{lemma:ends}.

By Theorem \ref{thm:CartanMatrixOfFullRank}, we may realize $W_m$ as a Zariski-dense subgroup of $\mathrm{SL}_{5+m}^\pm(\mathbb{Z})$. Since~$\Gamma_m$ is an infinite normal subgroup of a finite-index subgroup of $W_m$, it then follows from simplicity of $\mathrm{SL}_{5+m}(\mathbb{R})$ that the Zariski-closure of $\Gamma_m$ in $\mathrm{SL}_{5+m}^\pm(\mathbb{R})$ contains $\mathrm{SL}_{5+m}(\mathbb{R})$, so that ${\Gamma_m\cap\SL_{5+m}(\mathbb{Z})}$ satisfies the conclusion of the theorem.
\end{proof}

\begin{remark}
We describe another construction of a one-ended Zariski-dense witness $\Gamma_n$ to incoherence of $\mathrm{SL}_n(\mathbb{Z})$ for each $n \geq 120$ that uses forthcoming work of Fisher--Italiano--Kielak~\cite{fisher2025virtual}. The group~$W$ generated by the reflections in the sides of the right-angled $120$-cell in $\mathbb{H}^4$ virtually admits a map onto $\mathbb{Z}$ with finitely generated kernel $\Gamma$ \cite{jankiewicz2021virtually, kielak2020residually}. It is known however that such~$\Gamma$ cannot be finitely presented; see, for instance, \cite[Prop.~14]{isenrich2021hyperbolic}. Moreover, such~$\Gamma$ is one-ended by \cite{fisher2025virtual}. By Theorem~\ref{thm:virtually-Zariski-dense}, for each $n \geq 120$, there is a finite-index subgroup $\Delta_n$ of $W$ and a Zariski-dense faithful representation $\rho_n : \Delta_n \rightarrow \mathrm{SL}_n(\mathbb{Z})$. We may now take $\Gamma_n = \rho_n(\Delta_n \cap \Gamma)$. 
\end{remark}

%\bibliographystyle{apalike}
%\bibliographystyle{alpha}
%\bibliography{zariskibib}
\printbibliography

\end{document}